\documentclass[12pt,reqno]{amsart}
\usepackage{tikz}
\usepackage[utf8]{inputenc}
\usepackage[T1]{fontenc}
\usepackage{indentfirst}
\usepackage{amsmath} 
\usepackage{amsfonts} 
\usepackage{amsthm} 
\usepackage{amscd}
\usepackage{color}
\usepackage{graphics}
\usepackage{icomma} 
\usepackage{amsmath} 
\usepackage{amsfonts} 
\usepackage{enumerate}
\usepackage{dsfont,  amssymb, stmaryrd, latexsym, mathrsfs}
\usepackage{xypic}
\usepackage{array}
\usepackage{longtable}
\usepackage{geometry}
\usepackage[frenchb,english]{babel}
\usepackage{color}
\usepackage{csquotes}
\usepackage{graphics}
\usepackage{multirow} 
\usepackage{multicol}
\theoremstyle{plain}
\newtheorem{them}{Theorem}[section]
\newtheorem{lem}{Lemma}[section]

\theoremstyle{remark}
\newtheorem{rema}{\bf Remark}[section]

\newtheorem{exams}{\textbf{Numerical Examples}}[section]

\def\NN{\mathds{N}}

\def\QQ{\mathbb{Q}}

\begin{document}
\selectlanguage{english}

\title[The $2$-rank of the class group...]{The $2$-rank of the class group of some real cyclic quartic number fields II}
\author[A. Azizi]{Abdelmalek Azizi}
\address{Abdelmalek Azizi: Mohammed First University, Sciences Faculty, Mathematics Department, Oujda, Morocco}
\email{abdelmalekazizi@yahoo.fr}
\author[M. Tamimi]{Mohammed Tamimi}
\address{Mohammed Tamimi: Mohammed First University, Sciences Faculty, Mathematics Department, Oujda, Morocco}
\email{med.tamimi@gmail.com}
\author[A. Zekhnini]{Abdelkader Zekhnini}
\address{Abdelkader Zekhnini: Mohammed First University, Pluridisciplinary Faculty, Mathematics Department, Nador, Morocco}
\email{zekha1@yahoo.fr}


\maketitle

\selectlanguage{english}

\begin{abstract}
In this paper,  we determine  the $2$-rank of the class group of certain classes of real cyclic quartic number fields. Precisely, we consider the case in which  the quadratic subfield is $\mathbb{Q}(\sqrt{l})$ with $l=2$ or a prime congruent to $1\,\pmod8$.
\end{abstract}
\section{\bf Introduction}
Let $K$ be a number field and $H$ its  $p$-class group, that is the Sylow $p$-subgroup of the ideal class group $\mathrm{Cl}(K)$ of $K$ in the wide sense, where $p$ is a prime integer. Class groups of  number fields have been studied for a long time, and
there are many very interesting  problems concerning their behavior.
A particular quantity of interest is the rank $r_p(H)$ of the $p$-class group $H$ defined as  the number of cyclic $p$-groups appearing in the decomposition of $\mathrm{Cl}(K)$, i.e. the dimension of the $\mathbb{F}_p$-vector space $\mathrm{Cl}(K)/\mathrm{Cl}(K)^p$, where $\mathbb{F}_p$ is the field of $p$ elements.

For $p=2$, many mathematicians  are interested in determining $r_2(H)$ and the power of $2$ dividing the class number of $K$.  Hasse \cite{Ha70}, Bauer \cite{Ba71} and others gave methods for determining the exact power of $2$ dividing the class number of a quadratic numbers field. These methods were developed by C. J. Parry and his co-authors  to determine $r_2(H)$ and the  power of $2$ dividing the class number of some  cyclic quartic  numbers fields $K$ having a quadratic subfield $k$ with odd class number (e.g., \cite{Parry 77, BrPa78, Parry 75, Pa75, Pa80}). To accomplish their task,  they needed a suitable genus theory convenient to their situation. Hence they showed that,  the theory firstly developed by Hilbert (\cite{Hi94}), assuming an imaginary base field $k$, can be adapted to the situation where $K$ is a totally imaginary quartic cyclic extension of a totally real quadratic subfield  $k$. In reality, this theory can be applied to any quartic number field $K$ having a quadratic subfield $k$ of odd class number (\cite{Pa80}).

The $2$-rank of the class group of any biquadratic number field  $K$  is determined (partially or totally) in many papers   (\cite{McPaRa95, McPaRa97, Parry 77, BrPa78, AM01, AM04}) up to the case: $K$ is a real quartic cyclic extension of the rational number field $\mathbb{Q}$.
This paper is devoted  to investigate the $2$-rank of the class group of this class of number fields. We will focus on the case where the unique quadratic subfield of $K$ is $k=\mathbb{Q}(\sqrt l)$ with $l$ is a prime congruent to $1\pmod8$ or $l=2$. Note that the case  $l$ congruent to $5\pmod8$ is studied  separately.

An outline of the paper is as follows. In  \S\ \ref{1} we summarize preliminary results on quartic cyclic number fields  and the ambiguous class numbers formula.   The main theorems   are presented in \S\ \ref{2} and \S\ \ref{3}. In \S\ \ref{4} we characterize all real cyclic quartic number fields $\mathbb{K}=\mathbb{Q}(\sqrt{n\epsilon_{0}\sqrt{l}})$, with $l$ is a prime congruent to $1\pmod8$ or $l=2$,  whose   $2$-class group is trivial, cyclic, of rank $2$ or $3$.
\section*{\bf Notations}
\noindent Throughout this paper, we adopt the following notations.
\begin{enumerate}[$\bullet$]
	\item $\mathbb{Q}$: the rational field.
	\item $l$: a prime integer congruent  to $1$ modulo $8$ or $l=2$ .
	\item $k=\mathbb{Q}(\sqrt{l})$: a quadratic field.
	\item $\epsilon_{0}$: the fundamental unit of $k$.
	\item $n$: a square-free positive integer relatively prime to $l$.
	\item $\delta=1$ or $2$.
	\item $d= n\epsilon_{0}\sqrt{l}$.
	\item $\mathbb{K}=k(\sqrt{d})$: a real quartic cyclic number field.
	\item $\mathcal{O}_k$ (resp. $\mathcal{O}_\mathbb{K}$): the ring of integers of $k$ (resp. $\mathbb{K}$).
	\item $H$: the 2-class group of  $\mathbb{K}$.
	\item $r_2(H)$: the rank of $H$.
   \item $\mathfrak{2}_{i},\;i \in \{1;2\} $: the prime ideal of $k$ above $2$ if $l\equiv1\pmod8$.
	\item $\mathbb{K^{*}}$, $k^*$: the nonzero elements of the fields $\mathbb{K }$ and $k$ respectively.
	\item $N_{\mathbb{K}/k}(\mathbb{K})$: elements of $k$ which are norm from $\mathbb{K}$.
	\item $p,\; q,\; p_{i},\; q_{j}$: odd prime integers.
	\item $(\frac{x,\,y}{p})$: quadratic norm residue symbol over $k$.
	\item $[\frac{\alpha}{\beta}]$:  quadratic residue symbol for $k$.
	\item $(\frac{a}{b})$: quadratic residue (Legendre) symbol.
	\item $(\frac{a}{b})_{_{4}}$: the rational $4$-th power residue symbol.
\end{enumerate}
\section{\bf Preliminary results}\label{1}
Let $K$ be a cyclic quartic extension of the rational number field $\mathbb{Q}$. By  \cite[Theorem 1]{wi87}, it is known  that $K$ can be expressed uniquely in the form $K=\mathbb{Q}(\sqrt{a(\ell+b\sqrt{\ell})})$, where $a, b, c$ and $\ell$ are integers satisfying the conditions: $a$ is odd and square-free, $\ell= b^2+c^2$ is  square free positive and relatively prime to $a$, with $b>0$ and $c>0$. Note that $K$  possesses a unique quadratic subfield $k=\mathbb{Q}(\sqrt\ell)$. Assuming  the class number of $k$ is odd and $N_{k/\mathbb{Q}}(\epsilon_0)=-1$, where $\epsilon_{0}$ is the fundamental unit of  $k$, then one can, by \cite{Parry 90, Xianke 84}, deduce that there exist an integer $n$ such that $K=\mathbb{Q}(\sqrt{n\epsilon_0\sqrt{\ell}})$ and:
\begin{center}
	$n=
	\left \{
	\begin{array}{rl}
	2a & \text{ if }\ell\equiv1\pmod4\text{ and }b\equiv 1 \pmod2\\
	a & \text{ otherwise}.
	\end{array}
	\right.
	$\end{center}
We need also the following theorem which  gives the conductor $f_K$ of $K$.
\begin{them}[\cite{wi87}]\label{wi87}
	The  conductor $f_K$  of the $($real or imaginary$)$ cyclic quartic field $K=\mathbb{Q}(\sqrt{a(\ell+b\sqrt{\ell}})$, where $a$ is an odd square-free integer, $\ell= b^2+c^2$ is a square-free positive integer relatively prime to $a$, with $b>0$ and $c>0$, is given by  $f_K=2^{e}|a|\ell$, where   $e$  is defined by:
	$$e=
	\left \{
	\begin{array}{l}
	3,\; if\; \ell\equiv 2 \pmod8\;
	or\; \ell \equiv 1\pmod4 \;and\; b \equiv 1 \pmod2,\\
	2, \;if\; \ell \equiv 1\pmod4,\; b \equiv 0\, \pmod2,\; a+b \equiv 3\pmod4,\\
	0,\; if\; \ell \equiv 1\pmod4, \;b \equiv 0\, \pmod2,\; a+b \equiv 1\pmod4.
	\end{array}
	\right.
	$$		
\end{them}
\noindent Recall that the extensions $K/\mathbb{Q}$ were  investigated  by  Hasse in a paper (\cite{Hasse 48})  prior to that of Leopoldt (\cite{Leo}) on  the arithmetic interpretation of the class number  of real abelian fields. They were also investigated  by M. N. Gras \cite[...]{Gras-M-N 80-81, Gr77, Gr79} and others. By \cite{Hasse 48}, the field $K$ can be real as it can be imaginary. Precisely, we have the following result that specify the number of real (resp. imaginary) cyclic quartic fields sharing the same conductor and the same quadratic subfield.
\begin{lem}[\cite{Hasse 48}]\label{6}
	For a given  square-free positive integer $\ell=p_{0}p_{1}\dots p_{m}$, where $p_i$ is a prime integer for all $i$, and for  a given conductor $f_K$, the number of real $($resp. imaginary$)$ cyclic quartic fields $K$ having the same  conductor  $f_K$ and the same quadratic subfield $k=\mathbb{Q}(\sqrt\ell)$ is equal to $2^{m}$ if $f_K\equiv 0\,\pmod8$. But if $f_K\not \equiv 0\,\pmod8$, then the number of real $($resp. imaginary$)$ cyclic quartic fields  $K$ with the conductor  $f_K$ and the same quadratic subfield $k=\mathbb{Q}(\sqrt\ell)$ is equal to $2^{m}$ or $0\ ($resp. $0$ or $2^{m})$.
	Moreover, a cyclic quartic field $K$  is real if and only if  $S=\prod_{p|f_K}s_{p}=+1$, where $s_{2}=-1$, $s_p=(-1)^{\frac{p-1}{e_{p}}}$ for odd prime integer  $p$ with $e_{p}$ is the ramification index of $p$ in  $K$.
\end{lem}
\begin{rema}
	Keeping  notations above, $K$ is then real if and only if $a> 0$ (equivalently $n>0$).
\end{rema}
\noindent The field $K$ also satisfies the following lemma.
\begin{lem}[\cite{Zink 66-67}]\label{lem: Zink 66-67}
	Let  $a, b$ and $c$ be positive integers,  $\ell= b^{2} + c^{2}$, with $a$ and $c$  odd, then $\mathbb{Q}(\sqrt{2a(\ell+b\sqrt{\ell}})=\mathbb{Q}(\sqrt{a(\ell+c\sqrt{\ell}})$.
\end{lem}
\noindent We end this section by  recalling  the numbers of ambiguous ideal classes of a quadratic extension $K/k$. This result will allow us to investigate the $2$-rank of the class group of $K$.
\begin{them}[\cite{Azizi 04, Parry 75}]\label{18}
	Let   $K/k$ be a cyclic extension of prime degree $p$. Denote by $A_{K/k}$ the number of ambiguous ideal classes of $K$ with respect to $k$, then:  $$ A_{K/k} = h(k) p^{\mu+r^{*}-(r+c+1)}$$
	\noindent with:\\
	$r$ the number of fundamental units of  $k$,\\
	$\mu$ the number of prime ideals in  $k$ $($finite or infinite$)$ which ramify in  $K$,\\
	$p^{r^{*}}=[N_{K/k}(K^{*})\cap E_{k}:E_{k}^p]$ with $E_{k}$ is the group of units of  $k$,\\
	$c=1$  if  $k$ contains  a primitive  $p$-th root of unity and  $c=0$ otherwise.\\
	Furthermore, for  $p=2$ and if the class number of $k$ is  odd, then the  $2$-rank of the class group of $K$ is
	$$\mu+r^{*}-(r+c+1).$$
\end{them}
\begin{rema}\label{17}
	Since  $-1$ and $\epsilon_{0}$ generate the unit group  of the quadratic field $k$, so
	\begin{enumerate}[$\bullet$]
		\item $r^*=0$,  if $-1, \epsilon_{0} \;\text{and}\; -\epsilon_{0}$ are not in $N_{K/k}(K^{*})$.
		\item $r^*=1$, if $(-1$ is  in $N_{K/k}(K^{*})$  and    $\epsilon_{0}$  is not) or  $(-1$ is not   in $N_{K/k}(K^{*})$ and  $\epsilon_{0}$ or $-\epsilon_{0}$ is).
		\item $r^*=2$, if  $-1$ and $\epsilon_{0}$ are  in  $N_{K/k}(K^{*})$.
	\end{enumerate}
\end{rema}
 To compute $r_2(H)$, the rank of the $2$-class group $H$ of  $K$, we will
distinguish many cases. For this, let $p_{1}$, $p_2$, $\cdots$, $p_{t}$, $q_{1}$, $\cdots$, $q_{s}$ be  positive prime integers. Put $\delta=1$ or $2$.
 \section{\bf The case $l\equiv1 \pmod8$}\label{2}
 Let  $l$ be  a prime integer congruent  to $1\pmod8$ and  $n$ a square free positive integer relatively prime to $l$.  Let  $\mathbb{K}=k(\sqrt{n\epsilon_{0}\sqrt{l}})$ and $k=\mathbb{Q}(\sqrt{l})$, where  $\epsilon_{0}$ is the fundamental unit of $k$.
 \begin{rema}  As $n$ is relatively prime to $l$, so the prime integers dividing  $n$  don't ramify in  $k$, and   $2$ splits in  $k$ since $l\equiv 1\,\pmod8$. Moreover,   $e_{l}$, the ramification index of $l$ in $\mathbb{K}$, is $4$. Thus $s_{l}=(-1)^{\frac{l-1}{4}}= 1$.
 \end{rema}
 \begin{rema}\label{5}
	In what follows, we consider  $l=b^2+c^2$ with  $b$ and $c$  two positive integers and $c$ odd.
	As $l\equiv1\pmod8$, so $b\equiv0\pmod4$ (e.g., \cite[page 2]{Cohn 85}). Recall, as mentioned in the beginning of section \ref{1}, that there exists an odd  square-free integer $a$ relatively prime to $l$ such that $\mathbb{K}=\mathbb{Q}(\sqrt{a(l+b\sqrt{l})})$ with
	    \begin{center}
		$a=
		\left \{
		\begin{array}{rl}
		\frac{n}{2} & \text{ if }l\equiv1\pmod4\text{ and }b\equiv 1 \pmod2,\\
		n & \text{ otherwise}.
		\end{array}
		\right.
		$\end{center}
\end{rema}
We also need the following lemma.
\begin{lem}[\cite{AzTa}]\label{20}
  Let   $l$ be  a prime integer congruent  to $1\pmod8$ and $\epsilon_{0}$ the fundamental unit of $k=\mathbb{Q}(\sqrt{l})$. Then $\epsilon_{0}\sqrt l\equiv1\pmod4$ in $k$.
\end{lem}
 \subsection{Case  $n = 1$ }
 \begin{them}\label{}
 	Let $\mathbb{K}=\mathbb{Q}(\sqrt{n\epsilon_{0}\sqrt{l}})$ be a real cyclic quartic number field, where $l\equiv1\pmod8$ is a positive prime integer, $n$ a square-free positive  integer relatively prime to $l$ and  $\epsilon_{0}$ the fundamental unit of  $k=\mathbb{Q}(\sqrt{l})$.   If $n=1$, then $r_2\left(H\right)= 0$.
 \end{them}
 \begin{proof} Since $a=n=1$,  $a+b \equiv 1+0\equiv 1 \pmod4$,  which implies, by Theorem \ref{wi87},  that $f_\mathbb{K} = l\not\equiv 0 \pmod8$. But $S =s_{l}= +1$,  then  Lemma \ref{6} ensures the existence of real number field $\mathbb{K}$ having as conductor $f_\mathbb{K}$ and as quadratic subfield $k$.
 Moreover, the only prime ideal of $k$ that ramifies in $\mathbb{K}$ is $(\sqrt{l})$, i.e. $\mu=1$.  To prove the theorem, we have to compute the integer $r^*$ (see Theorem \ref{18}) by applying Remark \ref{17}, and then we call the theorem \ref{18}.
 We have
 $$\begin{array}{ll}
 		\left(\frac{-1, \,  d}{\left(\sqrt{l}\right)}\right)=\left[\frac{-1}{\left(\sqrt{l}\right)}\right]=\left(\frac{-1}{l}\right)=1. \\
 		\left(\frac{\epsilon_{0}, \,  d}{(\sqrt{l})}\right)=\left[\frac{\epsilon_{0}}{(\sqrt{l})}\right]=\left[\frac{u}{(\sqrt{l})}\right]=\left(\frac{u}{l}\right)=\left(\frac{l}{u}\right)=\left(\frac{v^{2}l}{u}\right)=1,
 		\end{array}$$
indeed $\epsilon_{0}=u+v\sqrt{l}$, so  $-1=u^{2}-v^{2}l$ and thus $v^{2}l\equiv1\pmod u.$
 Hence $r^{*}=2$,  which implies that:   $$r_2\left(H\right) = \mu + r^{*} -3 =1+2-3=0.$$
 \end{proof}
 \subsection{Case  $n =2$}
 \begin{them}\label{}
 	Let $\mathbb{K}=\mathbb{Q}(\sqrt{n\epsilon_{0}\sqrt{l}})$ be a real cyclic quartic number field, where $l\equiv1\pmod8$ is a positive prime integer, $n$ a square-free  positive  integer relatively prime to $l$ and  $\epsilon_{0}$ the fundamental unit of  $k=\mathbb{Q}(\sqrt{l})$.     For $n= 2$, we have:
 	\begin{enumerate}[\rm1.]
 		\item If $\left( \frac{2}{l}\right) _{4}=(-1)^{\frac{l-1}{8}}$, then   $r_2\left(H\right) =2.$
 		\item If $\left( \frac{2}{l}\right) _{4}\neq (-1)^{\frac{l-1}{8}}$, then       $r_2\left(H\right) =  1.$
 	\end{enumerate}
 \end{them}
 \begin{proof} Since $n=2$, $n$ is even and  according to  Lemma \ref{lem: Zink 66-67}
 	 we get $\mathbb{K}=\mathbb{Q}(\sqrt{2a(l+b\sqrt{l}})=\mathbb{Q}(\sqrt{a(l+c\sqrt{l}})$ with $a=\frac{n}{2}=1$.
 	As $c\equiv 1\pmod2$, so, by Theorem \ref{wi87},  $f_\mathbb{K} = 2^{3}al =2^{3}l\equiv 0 \pmod8 $,  thus there exists as many real cyclic fields as imaginary ones having as a conductor $f_\mathbb{K}$ and as a quadratic subfield $k$. The prime ideals of  $k$  which ramify  in  $\mathbb{K}$ are $(\sqrt{l})$  and $\mathfrak{2}_{i}, \;i \in \{1, 2\} $,   where $2\mathcal{O}_{k}=\mathfrak{2}_{1}\mathfrak{2}_{2}$ is the  decomposition of $2$ in $k$,  i.e. $\mu=3$. We have
 	$$\begin{array}{ll}
 \left( \frac{-1, d}{\mathfrak{2}_{1}}\right) =\left( \frac{-1, d}{\mathfrak{2}_{2}}\right)=\left( \frac{-1, 2\epsilon_{0}\sqrt{l}}{\mathfrak{2}_{1}}\right)=\left( \frac{-1, 2}{\mathfrak{2}_{1}}\right)\left( \frac{-1, \epsilon_{0}\sqrt{l}}{\mathfrak{2}_{1}}\right)=\left( \frac{-1, 2}{2}\right)=1,\\
 	\text{indeed }\left( \frac{-1,\; \epsilon_{0}\sqrt{l}}{\mathfrak{2}_{1}}\right)=\left(\frac{\epsilon_{0}\sqrt{l}}{\mathfrak{2}_{1}}\right)^0=1\text{ since }\mathfrak{2}_{1} \text {don't ramify in } \mathbb{Q}(\sqrt{\epsilon_{0}\sqrt{l}})\ (\text{ see Lemma \ref{20}}).\\
 	\left( \frac{\epsilon_{0}, d}{\mathfrak{2}_{1}}\right) =\left( \frac{\epsilon_{0}, d}{\mathfrak{2}_{2}}\right)=\left( \frac{\epsilon_{0}, 2\epsilon_{0}\sqrt{l}}{\mathfrak{2}_{1}}\right)=\left( \frac{\epsilon_{0}, 2}{\mathfrak{2}_{1}}\right)\left( \frac{\epsilon_{0},
  \epsilon_{0}\sqrt{l}}{\mathfrak{2}_{1}}\right)=\left( \frac{\epsilon_{0}, 2}{\mathfrak{2}_{1}}\right)=\left( \frac{2}{l}\right) _{4}(-1)^{\frac{l-1}{8}}
 \ \text{ (see \cite{AM01})}.\\
 \text{Hence} \left( \frac{-\epsilon_{0}, d}{\mathfrak{2}_{1}}\right) =\left( \frac{-\epsilon_{0}, d}{\mathfrak{2}_{2}}\right)=\left( \frac{-1, d}{\mathfrak{2}_{1}}\right)\left( \frac{\epsilon_{0}, d}{\mathfrak{2}_{1}}\right) =\left(\frac{2}{l}\right) _{4}(-1)^{\frac{l-1}{8}}.
 \end{array}$$
Since $\left(\frac{-1,  \,   d}{(\sqrt{l})}\right)=\left(\frac{\pm\epsilon_{0}, \,   d}{(\sqrt{l})}\right)=1$  as in the first case, therefore	
 \begin{enumerate}[a.]
 	\item If $\left( \frac{2}{l}\right) _{4}=(-1)^{\frac{l-1}{8}}$, then $r^{*}=2$,   which implies that:   $$r_2\left(H\right) = \mu + r^{*} -3 =3+2-3=2.$$
 	 \item If $\left( \frac{2}{l}\right) _{4}\neq (-1)^{\frac{l-1}{8}}$, then $r^{*}=1$,   which implies that:   $$r_2\left(H\right) = \mu + r^{*} -3 =3+1-3=1.$$
 \end{enumerate}	
 \end{proof}
 \subsection{\textbf{Case  $n=\prod_{i=1}^tp_{i}$ and, for all $i$,  $p_i\equiv1\pmod4$}}
 \begin{them}\label{}
 	Let $\mathbb{K}=\mathbb{Q}(\sqrt{n\epsilon_{0}\sqrt{l}})$ be a real cyclic quartic number field, where $l\equiv1\pmod8$ is a positive prime integer, $n$ a square-free positive  integer relatively prime to $l$ and  $\epsilon_{0}$ the fundamental unit of  $k=\mathbb{Q}(\sqrt{l})$.
 	Let  $n=\prod_{i=1}^{i=t}p_{i}$ with $p_{i}\equiv 1\pmod4  $ for all $i\in\{1, \dots, t \}$ and $t$ is a positive integer.
 	\begin{enumerate}[\rm1.]
 		\item If, for all $i$, $(\frac{p_{i}}{l})=-1$, then  $r_2(H) =t$.
 		\item If, for all $i$, $(\frac{p_{i}}{l})= 1$, then:
 		\begin{enumerate}[\rm a.]
 			\item If $\left(\frac{p_i}{l}\right)_{_{4}}\neq \left(\frac{l}{p_i}\right)_{_{4}} $ for at least one $i\in\{1, \dots, t \}$, then $r_2\left(H\right) = 2t-1.$
 			\item If $\left(\frac{p_i}{l}\right)_{_{4}}= \left(\frac{l}{p_i}\right)_{_{4}} $ for all $i\in\{1, \dots, t \}$, then   $r_2\left(H\right) = 2t.$
 		\end{enumerate}
 	\end{enumerate}
 	Moreover, for $n=\prod_{i=1}^{i=t_{1}}p_{i}\prod_{j=1}^{i=t_{2}}q_{j}$ with   $(\frac{p_{i}}{l})=-(\frac{q_{j}}{l})=-1$ and $p_{i}\equiv q_{j}\equiv1\,\pmod4$ for all $i\in\{1, \dots, t_{1}\}$ and for all  $j\in\{1, \dots, t_{2}\}$, we have:
 	\begin{enumerate}[\rm a.]
 		\item If $\left(\frac{q_j}{l}\right)_{_{4}}\neq \left(\frac{l}{q_j}\right)_{_{4}} $ for at least one $j\in\{1, \dots, t_{2} \}$, then $r_2\left(H\right) = t_{1}+2t_{2}-1.$
 		\item If $\left(\frac{q_j}{l}\right)_{_{4}}= \left(\frac{l}{q_j}\right)_{_{4}} $ for all $j\in\{1, \dots, t_{2} \}$, then   $r_2\left(H\right) = t_{1}+2t_{2}.$
 	\end{enumerate}	
 \end{them}
 \begin{proof}  Note first that  $b\equiv0\pmod2$,  and since $l\equiv1\pmod8$, so $b\equiv0\pmod4$ (see Remark \ref{5}). On the other hand,  $n=\prod_{i=1}^{i=t}p_{i}$, with  $p_{i}\equiv1\pmod4$ for all  $i=1, \dots, t$, then  $n=a=\prod_{i=1}^{i=t}p_{i}\equiv1\,\pmod4$, so  $a+b\equiv 1+0\equiv1\,\pmod4$. Thus, by Theorem \ref{wi87}, $f_\mathbb{K}=al=\prod_{i=1}^{i=t}p_{i}l$, which implies that  $f_\mathbb{K}\not \equiv 0\,\pmod8$. Thus    $\mathbb{K}$ is either real or imaginary cyclic quartic field. As  the ramification index in $\mathbb{K}$ of each prime $p_i$ is   $e_{p_{i}}=2$, then  $\frac{p_{i}-1}{2}\equiv0\pmod2$, this  implies that $s_{p_{i}}=(-1)^{\frac{p_{i}-1}{2}}=+1$. Hence
 	$S=\prod_{q|f}s_{q}=\prod_{i=1}^{i=t}s_{p_{i}}s_{l}=+1$, and Lemma \ref{6} ensures the existence of real number field $\mathbb{K}$,  for all nonzero positive integer $t$, having as conductor $f_\mathbb{K}$ and as quadratic subfield $k$.
 	\begin{enumerate}[\rm1.]
 		\item If $(\frac{p_{i}}{l})=-1$,  for all $i=1, \dots, t$,  then the prime ideals of $k$ which ramify in $\mathbb{K}$ are $(\sqrt{l})$ and the prime ideals $\mathfrak{p}_{i}$,  $i=1,  \dots,  t$,  where   $\mathfrak{p}_{i}$ is the prime ideal of $k$ above $p_i$, this implies that the number of prime ideals of $k$ ramifying in $\mathbb{K}$ is $\mu= t + 1$. Hence:
 		$$\begin{array}{ll}
 		\left(\frac{-1, \,  d}{\mathfrak{p}_{i}}\right)=\left[\frac{-1}{\mathfrak{p}_{i}}\right]=\left(\frac{-1}{p_i}\right)=1,  \text{ for all } i=1, \dots,  t. \\
 		\left(\frac{\epsilon_{0}, \,  d}{\mathfrak{p}_{i}}\right)=\left[\frac{\epsilon_{0}}{\mathfrak{p}_{i}}\right]=\left(\frac{-1}{p_{i}}\right)=1,
 		\text{ for all }  i=1, \dots,  t.
 		\end{array}$$
The values $\left(\frac{-1,  \,   d}{(\sqrt{l})}\right)$ and  $\left(\frac{\epsilon_{0}, \,   d}{(\sqrt{l})}\right)$ are computed as in the first case.
 		So $r^{*}=2$, from which we infer that:  $$r_2(H) = \mu + r^{*} -3 = t + 1 + 2-3=t.$$
 		\item If $(\frac{p_{i}}{l})=1$,  for all $i=1, \dots, t$,  then the prime ideals of $k$ which ramify in $\mathbb{K}$ are $(\sqrt{l})$,  and the prime ideals $\wp_{i}$ and $\bar{\wp}_{i}$   with $p_i\mathcal{O}_k=\wp_{i}\bar{\wp}_{i}$,  $i=1, \dots, t$, in this case $\mu= 2t + 1$.  Hence
 		$$\begin{array}{ll}
 \left(\frac{-1, d}{\wp_{i}}\right)=\left(\frac{-1, d}{\bar{\wp}_{i}}\right)=\left[\frac{-1}{\wp_{i}}\right]=\left(\frac{-1}{p_{i}}\right)=1, \text{ for all } i=1, \dots, t.\\
 		\left(\frac{\epsilon_{0}, d}{\wp_{i}}\right) =\left(\frac{\epsilon_{0},d}{\bar{\wp}_{i}}\right)=\left[ \frac{\epsilon_{0}}{\wp_{i}}\right]=\left(\frac{\epsilon_{0}, p}{\wp_{i}}\right)=\left(\frac{p_i}{l}\right)_{4}\left(\frac{l}{p_i}\right)_{4},\text{ for all } i\in\{1, \dots, t\}\, $(\text{see}\, \cite{AM01})$.\\
       \text{Hence}\, \left(\frac{-\epsilon_{0}, p}{\wp_{i}}\right)=\left(\frac{-\epsilon_{0},d}{\bar{\wp}_{i}}\right)=\left(\frac{p_i}{l}\right)_{4}\left(\frac{l}{p_i}\right)_{4}	
 \end{array}$$
 	   Since $\left(\frac{-1, \,  d}{(\sqrt{l})}\right)=\left(\frac{\pm\epsilon_{0},\,  d}{(\sqrt{l})}\right)=1$ as above, we infer that:
 		\begin{enumerate}[\rm a.]
 			\item If $\left(\frac{p_i}{l}\right)_{_{4}}\neq \left(\frac{l}{p_i}\right)_{_{4}} $ for at least one $i\in\{1, \dots, t \}$, then $r^{*}=1$ and $$r_2\left(H\right) = \mu + r^{*} -3 = 2t + 1 + 1-3=2t-1.$$
 			\item If $\left(\frac{p_i}{l}\right)_{_{4}}= \left(\frac{l}{p_i}\right)_{_{4}} $ for all $i\in\{1, \dots, t \}$, then $r^{*}=2$ and $$r_2\left(H\right) = \mu + r^{*} -3 = 2t + 1 + 2-3=2t.$$
 		\end{enumerate}		
 	\end{enumerate}
 	If $n=\prod_{i=1}^{i=t_{1}}p_{i}\prod_{j=1}^{j=t_{2}}q_{j}$ with $p_i\equiv q_j\equiv1 \pmod4$ and  $\left(\frac{p_{i}}{l}\right)=-\left(\frac{q_{j}}{l}\right)=-1$ for all $ i\in\{1, \dots, t_{1}\}$ and for all  $j\in\{1, \dots, t_{2}\}$,  then according to the two cases above we have:
 	\begin{enumerate}[\rm a.]
 		\item If $\left(\frac{q_j}{l}\right)_{_{4}}\neq \left(\frac{l}{q_j}\right)_{_{4}} $ for at least one $j\in\{1, \dots, t_{2} \}$, then $r^{*}=1$ and $$r_2\left(H\right) = \mu + r^{*} -3 = t_{1}+2t_{2} + 1 + 1-3=t_{1}+2t_{2}-1.$$
 		\item If $\left(\frac{q_j}{l}\right)_{_{4}}= \left(\frac{l}{q_j}\right)_{_{4}} $ for all $j\in\{1, \dots, t_{2} \}$, then $r^{*}=2$ and $$r_2\left(H\right) = \mu + r^{*} -3 = t_{1}+2t_{2} + 1 + 2-3=t_{1}+2t_{2}.$$
 	\end{enumerate}
 \end{proof}
 \subsection{\textbf{Case  $n=2\prod_{i=1}^tp_{i}$,  and for all $i$,   $p_i\equiv1\pmod4$}}
 \begin{them}\label{}
 	Let $\mathbb{K}=\mathbb{Q}(\sqrt{n\epsilon_{0}\sqrt{l}})$ be a real cyclic quartic number field,  where $l\equiv1\pmod8$ is a positive prime integer,  $n$ a square-free positive  integer relatively prime to $l$ and  $\epsilon_{0}$ the fundamental unit of  $k=\mathbb{Q}(\sqrt{l})$.
 	Let  $n=2\prod_{i=1}^{i=t}p_{i}$ with $p_{i}\equiv 1\pmod4  $ for all $i\in\{1,  \dots,  t \}$ and $t$ is a positive integer.
 	\begin{enumerate}[\rm1.]
 		\item If,  for all $i$,  $(\frac{p_{i}}{l})=-1$,  then:
 			\begin{enumerate}[\rm a.]
 			\item If $\left( \frac{2}{l}\right) _{4}=(-1)^{\frac{l-1}{8}}$,  then  $r_2(H) = t+2.$
 			\item If $\left( \frac{2}{l}\right) _{4}\neq (-1)^{\frac{l-1}{8}}$,  then     $r_2(H)= t+1.$
 	    	\end{enumerate}
 		\item If,  for all $i$,  $(\frac{p_{i}}{l})= 1$,  then
 		\begin{enumerate}[\rm a.]
 			\item If $\left( \frac{2}{l}\right) _{4}=(-1)^{\frac{l-1}{8}}$ and $\left(\frac{p_i}{l}\right)_{_{4}}=\left(\frac{l}{p_i}\right)_{_{4}} $ for all $ i\in\{1,  \dots,  t\}$,   then $r_2(H) =2t+2.$
 			\item If $\left( \frac{2}{l}\right) _{4}\neq (-1)^{\frac{l-1}{8}}$ or $\left(\frac{p_i}{l}\right)_{_{4}}\neq \left(\frac{l}{p_i}\right)_{_{4}} $ for at least one $ i\in\{1,  \dots,  t\}$,  then $r_2(H) = 2t+1.$
 		\end{enumerate}
 	\end{enumerate}
  Moreover,  if $n=2\prod_{i=1}^{i=t_{1}}p_{i}\prod_{j=1}^{j=t_{2}}q_{j}$ with $p_i\equiv q_j \equiv1\pmod4$ and  $(\frac{p_{i}}{l})=-(\frac{q_{j}}{l})=-1$  for all $i\in\{1,  \dots,  t_{1}\}$ and for all  $j\in\{1,  \dots,  t_{2}\}$,  then:
  \begin{enumerate}[\rm a.]
  	\item If $\left( \frac{2}{l}\right) _{4}=(-1)^{\frac{l-1}{8}}$ and $\left(\frac{q_j}{l}\right)_{_{4}}=\left(\frac{l}{q_j}\right)_{_{4}} $ for all $ j\in\{1,  \dots,  t_{2}\}$,  then $r_2(H) =t_{1}+2t_{2}+2.$
  	\item If $\left( \frac{2}{l}\right) _{4}\neq (-1)^{\frac{l-1}{8}}$ or $\left(\frac{q_j}{l}\right)_{_{4}}\neq \left(\frac{l}{q_j}\right)_{_{4}} $ for at least one $ j\in\{1,  \dots,  t_{2}\}$, then $r_2(H) =t_{1}+ 2t_{2}+1.$
  \end{enumerate}
 \end{them}
 \begin{proof} 	
 	For $n=2\prod_{i=1}^{i=t}p_{i}$, with  $p_{i}\equiv1\pmod4$ for all  $i=1,  \dots,  t$,  we have,  according to  Lemma \ref{lem: Zink 66-67}, $\mathbb{K}=\mathbb{Q}(\sqrt{2a(l+b\sqrt{l}})=\mathbb{Q}(\sqrt{a(l+c\sqrt{l}})$ with $a=\frac{n}{2}=\prod_{i=1}^{i=t}p_{i}$. As  $c\equiv 1\pmod2$ and $a=\frac{n}{2}=\prod_{i=1}^{i=t}p_{i}\equiv 1\pmod4$,  so,  by Theorem \ref{wi87},   $f_{\mathbb{K}}= 2^{3}al=2^{3}l\prod_{i=1}^{i=t}p_{i}\equiv0\, \pmod8$,  this implies,  by Lemma \ref{6},  that  there are as many real cyclic  fields as  imaginary ones $\mathbb{K}$  having as a conductor $f_\mathbb{K}$ and as quadratic subfield $k$.
 	\begin{enumerate}[\rm1.]
 		\item If $(\frac{p_{i}}{l})=-1$,   for all $i=1,  \dots,  t$,   then the prime ideals of $k$ which ramify in $\mathbb{K}$ are $(\sqrt{l})$,    $\mathfrak{2}_{i}, \;i \in \{1, 2\}$,  and the prime ideals $\mathfrak{p}_{i}$,   $i=1,   \dots,   t$,   where    $\mathfrak{p}_{i}$ is the prime ideal of $k$ above  $p_i$ and $2\mathcal{O}_{k}=\mathfrak{2}_{1}\mathfrak{2}_{2}$ is the decomposition of $2$ in $k$,  this implies that the number of prime ideals of $k$ ramifying in $\mathbb{K}$ is $\mu= t + 3$. As, by Lemma \ref{20}, $\epsilon_{0}\sqrt{l}\equiv1\pmod4$, so $\prod_{i=1}^{i=t}p_{i}\epsilon_{0}\sqrt{l}\equiv1\pmod4$. Thus $\mathfrak{2}_{1}$ is unramified in $\mathbb{Q}(\sqrt{\prod_{i=1}^{i=t}p_{i}\epsilon_{0}\sqrt{l}})$, then   $\left( \frac{-1,\, \prod_{i=1}^{i=t}p_{i}\epsilon_{0}\sqrt{l}}{\mathfrak{2}_{1}}\right)=\left( \frac{\epsilon_{0}, \prod_{i=1}^{i=t}p_{i}\epsilon_{0}\sqrt{l}}{\mathfrak{2}_{1}}\right)=1$. Therefore,
 		$\begin{array}{ll}
 \left( \frac{-1, d}{\mathfrak{2}_{1}}\right) =\left( \frac{-1, d}{\mathfrak{2}_{2}}\right)=\left( \frac{-1, 2\prod_{i=1}^{i=t}p_{i}\epsilon_{0}\sqrt{l}}{\mathfrak{2}_{1}}\right)=\left( \frac{-1, 2}{\mathfrak{2}_{1}}\right)\left( \frac{-1, \prod_{i=1}^{i=t}p_{i}\epsilon_{0}\sqrt{l}}{\mathfrak{2}_{1}}\right)=\left( \frac{-1, 2}{2}\right)=1,\\
 	\left( \frac{\epsilon_{0}, d}{\mathfrak{2}_{1}}\right) =\left( \frac{\epsilon_{0}, d}{\mathfrak{2}_{2}}\right)=\left( \frac{\epsilon_{0}, 2\prod_{i=1}^{i=t}p_{i}\epsilon_{0}\sqrt{l}}{\mathfrak{2}_{1}}\right)=\left( \frac{\epsilon_{0}, 2}{\mathfrak{2}_{1}}\right)\left( \frac{\epsilon_{0}, \prod_{i=1}^{i=t}p_{i}\epsilon_{0}\sqrt{l}}{\mathfrak{2}_{1}}\right)=\left( \frac{\epsilon_{0}, 2}{\mathfrak{2}_{1}}\right)=\left( \frac{2}{l}\right) _{4}(-1)^{\frac{l-1}{8}}\\
 	\text{Hence} \left( \frac{-\epsilon_{0}, d}{\mathfrak{2}_{1}}\right) =\left( \frac{-\epsilon_{0}, d}{\mathfrak{2}_{2}}\right)=\left( \frac{2}{l}\right) _{4}(-1)^{\frac{l-1}{8}}
 	.\end{array}$
 	 		The  values $\left(\frac{-1,  \,   d}{(\sqrt{l})}\right)$,  $\left(\frac{\pm\epsilon_{0}, \,   d}{(\sqrt{l})}\right)$,  $ \left( \frac{-1, d}{\mathfrak{p}_{i}}\right) $,  and $ \left( \frac{\pm\epsilon_{0}, d}{\mathfrak{p}_{i}}\right) $ for   $i\in\{1,  \dots,  t \}$   are computed as above. So:
 	\begin{enumerate}[\rm a.]
 		\item If $\left( \frac{2}{l}\right) _{4}=(-1)^{\frac{l-1}{8}}$, then $r^{*}=2$,   which implies that:   $$r_2(H) = \mu + r^{*} -3 = t + 3 +2 -3=t+2.$$
 		\item If $\left( \frac{2}{l}\right) _{4}\neq (-1)^{\frac{l-1}{8}}$, then $r^{*}=1$,   which implies that:   $$r_2(H) = \mu + r^{*} -3 = t + 3 +1 -3=t+1.$$
 	\end{enumerate}
 		\item If $(\frac{p_{i}}{l})=1$,   for all $i=1,  \dots,  t$,   then the prime ideals of $k$ which ramify in $\mathbb{K}$ are $(\sqrt{l})$,    $\mathfrak{2}_{i}, \;i \in \{1, 2\} $ and the prime ideals $\wp_{i}$ and $\bar{\wp}_{i}$   with $p_i\mathcal{O}_k=\wp_{i}\bar{\wp}_{i}$,   $i=1,  \dots,  t$,  in this case $\mu= 2t + 3$. So using the results of  above cases we infer that:
 	\begin{enumerate}[\rm a.]
 		\item If $\left( \frac{2}{l}\right) _{4}=(-1)^{\frac{l-1}{8}}$ and $\left(\frac{p_i}{l}\right)_{_{4}}=\left(\frac{l}{p_i}\right)_{_{4}} $ for all $ i\in\{1,  \dots,  t\}$  then $r^{*}=2$,  so:   $$r_2(H) = \mu + r^{*} -3 = 2t + 3 +2 -3=2t+2.$$
 		\item If $\left( \frac{2}{l}\right) _{4}\neq (-1)^{\frac{l-1}{8}}$ or $\left(\frac{p_i}{l}\right)_{_{4}}\neq \left(\frac{l}{p_i}\right)_{_{4}} $ for at least one $ i\in\{1,  \dots,  t\}$ then $r^{*}=1$,   which implies that:   $$r_2(H) = \mu + r^{*} -3 = 2t + 3 +1 -3=2t+1.$$
 	\end{enumerate}
 	\end{enumerate}
 	If $n=2\prod_{i=1}^{i=t_{1}}p_{i}\prod_{j=1}^{j=t_{2}}q_{j}$ with $p_i\equiv q_j \equiv 1\pmod4$ and $\left(\frac{p_{i}}{l}\right)=-\left(\frac{q_{j}}{l}\right)=-1$ for all $ i\in\{1,  \dots,  t_{1}\}$ and for all  $j\in\{1,  \dots,  t_{2}\}$,   then according to the two cases above:
 	\begin{enumerate}[\rm a.]
 		\item If $\left( \frac{2}{l}\right) _{4}=(-1)^{\frac{l-1}{8}}$ and $\left(\frac{q_j}{l}\right)_{_{4}}=\left(\frac{l}{q_j}\right)_{_{4}} $ for all $ j\in\{1,  \dots,  t_{2}\}$,  then $r^{*}=2$ and   $r_2\left(H\right)= t_{1}+2t_{2}+2\cdot$
 		\item If $\left( \frac{2}{l}\right) _{4}\neq (-1)^{\frac{l-1}{8}}$ or $\left(\frac{q_j}{l}\right)_{_{4}}\neq \left(\frac{l}{q_j}\right)_{_{4}} $ for at least one $ j\in\{1,  \dots,  t_{2}\}$, then $r^{*}=1$ and   $r_2\left(H\right)= t_{1}+2t_{2}+1\cdot$
 	\end{enumerate}
 \end{proof}
 \subsection{\textbf{Case  $n =\delta\prod_{i=1}^{i=t}p_{i}$ with $t$ is odd,  and for all $i$,  $p_i\equiv3\pmod4$}}\label{sec35}
 \begin{them}\label{19}
 	Let $\mathbb{K}=\mathbb{Q}(\sqrt{n\epsilon_{0}\sqrt{l}})$ be a real cyclic quartic number field,  where $l\equiv1\pmod8$ is a positive prime integer,  $n$ a square-free positive  integer relatively prime to $l$ and  $\epsilon_{0}$ the fundamental unit of  $k=\mathbb{Q}(\sqrt{l})$. Assume  $n =\delta\prod_{i=1}^{i=t}p_{i}$,   $p_{i}\equiv3\pmod4$ for all  $i=1,  \dots,  t$       and $t$ is a positive odd integer.
 	\begin{enumerate}[\rm1.]
 		\item If,  for all $i$,   $\left(\frac{p_{i}}{l}\right)=-1$,   then  $r_2\left(H\right) =t$.
 		\item If,  for all $i$,   $\left(\frac{p_{i}}{l}\right)= 1$,   then  $r_2\left(H\right) =2t$.
 	\end{enumerate}	
 	Moreover,   if $n=\prod_{i=1}^{t_{1}}p_{i}\prod_{j=1}^{t_{2}}q_{j}$,  where  $ p_{i}\equiv q_{j}\equiv3\pmod4$ and  $\left(\frac{p_{i}}{l}\right)=-\left(\frac{q_{j}}{l}\right)=-1, $ for all  $i\in\{1,  \dots,  t_{1}\}$,   and for all $j\in\{1,  \dots,  t_{2}\}$ with $t_{1}+t_{2}$ is odd,   then  $r_2\left(H\right)=  t_{1} +2t_{2}$.
 \end{them}
 \begin{proof}
 Assume $n= \prod_{i=1}^{i=t} p_{i}$, so	since $t$ is odd,   we have $a=\prod_{i=1}^{i=t}p_{i}\equiv 3\,  \pmod4$,  so $a+b \equiv 3+0 \equiv 3\,  \pmod4$. Thus,  by Theorem \ref{wi87},   $f_\mathbb{K} = 2^{2}al=2^{2}l\prod_{i=1}^{i=t}p_{i}\not\equiv 0 \pmod8$. But $S =s_{2}s_{l}\prod_{i=1}^{i=t}s_{p_{i}} = +1$. Indeed for $i=1,  \dots,  t$,  we have $e_{p_{i}}=2 $ and $p_{i}\equiv 3\,  \pmod4$,   then $\frac{p_{i}-1}{2}\equiv1\pmod2$,  so $s_{p_{i}}=-1$. Therefore,  the  real number field $\mathbb{K}$  having as a conductor $f_\mathbb{K}$ and as quadratic subfield $k$ exists.\\
 	If  $n= 2\prod_{i=1}^{i=t} p_{i}$,   $p_{i}\equiv 3\pmod4$ for all $i=1,  \dots,  t$ with $t$ odd,  then  by  Lemma \ref{lem: Zink 66-67}	 we get $\mathbb{K}=\mathbb{Q}(\sqrt{2a(l+b\sqrt{l}})=\mathbb{Q}(\sqrt{a(l+c\sqrt{l}})$ with $a=\frac{n}{2}=\prod_{i=1}^{i=t}p_{i}$.	As $c\equiv 1\pmod2$ and $l\equiv1\pmod4$,  so $f_\mathbb{K}= 2^{3}l \prod_{i=1}^{i=t} p_{i} \equiv 0 \pmod8$, then there exist as many real cyclic quartic number fields as imaginary ones sharing  the conductor $f_\mathbb{K}$ and the quadratic subfield $k$.
 	\begin{enumerate}[\rm1.]
 		\item If $\left(\frac{p_{i}}{l}\right)=-1$, for all  $i\in\{1,  \dots,  t\}$,  then  the prime ideals of  $k$ which ramify in  $\mathbb{K}$ are  $\mathfrak{2}_{i}, \;i \in \{1, 2\} $,   $\mathfrak{p}_{i}$ and  $(\sqrt{l}) $,  where $\mathfrak{p}_{i}$ is the prime ideal of $k$ above $p_i$ and $2\mathcal{O}_{k}=\mathfrak{2}_{1}\mathfrak{2}_{2}$ the decomposition of $2$ in $k$.
 		\begin{enumerate}[\rm a.]
 		\item For the case $n =\prod_{i=1}^{i=t}p_{i}$,  $i\in\{1,  \dots,  t\}$,  we have:
 		$$\begin{array}{ll}
 		\left( \frac{-1, d}{\mathfrak{2}_{1}}\right) =\left( \frac{-1, d}{\mathfrak{2}_{2}}\right)=\prod_{i=1}^{i=t}\left( \frac{-1, p_{i}}{\mathfrak{2}_{1}}\right)\left( \frac{-1, \epsilon_{0}\sqrt{l}}{\mathfrak{2}_{1}}\right)=\prod_{i=1}^{i=t}\left( \frac{-1, p_{i}}{\mathfrak{2}_1}\right)=(-1)^{t}=-1.\\
 		\left(\frac{\epsilon_{0}, \,  d}{\mathfrak{p}_{i}}\right)=\left[\frac{\epsilon_{0}}{\mathfrak{p}_{i}}\right]=\left(\frac{-1}{p_{i}}\right)= -1,\;
 		\text{so}\, \left(\frac{-\epsilon_{0}, \,  d}{\mathfrak{p}_{i}}\right)=-1.
 		\end{array}$$
 	 	\item For the case $n =2\prod_{i=1}^{i=t}p_{i}$,  $i\in\{1,  \dots,  t\}$,  we have:
 	$$\begin{array}{ll}
 \left( \frac{-1, d}{\mathfrak{2}_{1}}\right) =\left( \frac{-1, d}{\mathfrak{2}_{2}}\right)=\left( \frac{-1, 2\prod_{i=1}^{i=t}p_{i}\epsilon_{0}\sqrt{l}}{\mathfrak{2}_{1}}\right)=\left( \frac{-1, 2}{\mathfrak{2}_{1}}\right)\prod_{i=1}^{i=t}\left( \frac{-1, p_{i}}{\mathfrak{2}_{1}}\right)\left( \frac{-1, \epsilon_{0}\sqrt{l}}{\mathfrak{2}_{1}}\right)=(-1)^{t}=-1,\\
 	\text{As above we have} \left(\frac{\pm\epsilon_{0}, \,  d}{\mathfrak{p}_{i}}\right)= -1.
 	 \end{array}$$
 \end{enumerate}
 		Hence for the two cases, the units $-1, \epsilon_{0}$ and $-\epsilon_{0}$ are not norms in $\mathbb{K }$, then we have: $r^{*}=0$,   which implies:   $$r_2\left(H\right) = \mu + r^{*} -3 =t+3+0-3=t.$$
 		\item If $\left(\frac{p_{i}}{l}\right)=1$ for all  $i\in\{1,  \dots,  t\}$,  then  the prime ideals of  $k$ which ramify in  $\mathbb{K}$ are  $\mathfrak{2}_{i}, \;i \in \{1, 2\} $, $(\sqrt{l})$,   $\wp_{i}$ and $\bar{\wp}_{i}$,   where $p_i\mathcal{O}_k=\wp_{i}\bar{\wp}_{i}$,  $i=1,  \dots,  t$ and $2\mathcal{O}_k=\mathfrak{2}_{1}\mathfrak{2}_{2}$.
 		\begin{enumerate}[\rm a.]
 		\item For the case $n =\prod_{i=1}^{i=t}p_{i}$,  $i\in\{1,  \dots,  t\}$,  we have for all $i\in\{1,  \dots,  t\}$:
 		$$\begin{array}{ll}
 		\left(\frac{-1, \,  d}{\wp_{i}}\right)=\left(\frac{-1,  d}{\bar{\wp}_{i}}\right)=\left[\frac{-1}{\wp_{i}}\right]=\left(\frac{-1}{p_{i}}\right)=-1.\\
 		\left(\frac{\epsilon_{0},  d}{\wp_{i}}\right)=\left[\frac{\epsilon_{0}}{\wp_{i}}\right], \text{ for all } i=1,  \dots,  t.
 		\end{array}$$
 	 To compute the last unity,  put $p_i^{h_{0}}=\wp_{i}\bar{\wp}_{i}$ and  $\wp_{i}= a_{i} +b_{i} \sqrt{l}$ and $\bar{\wp}_{i}= a_{i} -b_{i} \sqrt{l}, $ for all   $i$. According to \cite{BrPa78}  we have
 	 $\left[\frac{\epsilon_{0} \sqrt{l}}{\wp_{i}}\right]=\left[\frac{\epsilon_{0} \sqrt{l}}{\bar{\wp}_{i}}\right]= \left(\frac{p_i}{l}\right)_{4}.$  Thus
 	 $$\begin{array}{ll}
 	 \left[\frac{\epsilon_{0} }{\wp_{i}}\right]= \left(\frac{p_i}{l}\right)_{4}\left[\frac{ \sqrt{l}}{\wp_{i}}\right].\end{array}$$
 	 On the other hand,
 	 $$\begin{array}{ll}
 	 \left[\frac{ \sqrt{l}}{\wp_{i}}\right]=\left[\frac{ b_{i}^{2}\sqrt{l}}{\wp_{i}}\right] =\left[\frac{b_{i}(-a_{i}+a_{i} + b_{i} \sqrt{l})}{\wp_{i}}\right]=\left[\frac{ -a_{i}b_{i}}{\wp_{i}}\right]=-\left(\frac{a_{i}}{p_i}\right)\left(\frac{b_{i}}{p_i}\right).\end{array}$$
 	 As $p_i^{h_{0}}=a_{i}^{2} - b_{i}^{2}l$,  so  $b_{i}^{2}l\equiv a_{i}^{2}\,   \pmod {p_i}$. Since $l$ and  $p_i$ are relatively prime,  then   $b_{i}^{2}l^{2}\equiv la_{i}^{2}\,   \pmod {p_i}$,   so  $\left(\frac{b_{i}}{p_i}\right)=\left(\frac{b_{i}l}{p_i}\right)=\left(\frac{b_{i}^{2}l^{2}}{p_i}\right)_{_{4}}=\left(\frac{la_{i}^{2}}{p_i}\right)_{_{4}}=
 	 \left(\frac{l}{p_i}\right)_{4}\left(\frac{a_{i}}{p_i}\right)$.
 	 Finally,   $$\begin{array}{ll}
 	 \left[\frac{\epsilon_{0} }{\wp_{i}}\right]=- \left(\frac{p_i}{l}\right)_{_{4}}\left(\frac{ a_{i}}{p_i}\right)\left(\frac{ b_{i}}{p_i}\right)=-\left(\frac{p_i}{l}\right)_{_{4}}\left(\frac{ a_{i}}{p_i}\right)\left(\frac{l}{p_i}\right)_{_{4}}\left(\frac{ a_{i}}{p_i}\right)=-\left(\frac{p_i}{l}\right)_{_{4}}\left(\frac{l}{p_i}\right)_{_{4}}.\end{array}$$
 	 Proceeding similarly,   we get   $\left[\frac{\epsilon_{0} }{\bar{\wp_{i}}}\right]=  \left(\frac{p_i}{l}\right)_{_{4}}\left(\frac{l}{p_i}\right)_{_{4}}$ using the fact $\left[\frac{ \sqrt{l}}{\bar{\wp}_{i}}\right]=-\left[\frac{ -b_{i}^{2}\sqrt{l}}{\bar{\wp}_{i}}\right]$.\\
 	For the unit $-\epsilon_{0}$ we have: $\left(\frac{-\epsilon_{0},  d}{\wp_{i}}\right)=-\left(\frac{-\epsilon_{0},  d}{\bar{\wp}_{i}}\right)=\left(\frac{p_i}{l}\right)_{_{4}}\left(\frac{l}{p_i}\right)_{_{4}}$
 		\item For the case $n =2\prod_{i=1}^{i=t}p_{i}$,  $i\in\{1,  \dots,  t\}$,
 		Similarly we have for all $i\in\{1,  \dots,  t\}$:
 		$\left(\frac{-1, \,  d}{\wp_{i}}\right)=-1$, $\left(\frac{\epsilon_{0},  d}{\wp_{i}}\right)\neq\left(\frac{\epsilon_{0},  d}{\bar{\wp}_{i}}\right)$ and  $\left(\frac{-\epsilon_{0},  d}{\wp_{i}}\right)\neq\left(\frac{-\epsilon_{0},  d}{\bar{\wp}_{i}}\right)$ 	
 		\end{enumerate}
 	Hence for the two cases $-1, \epsilon_{0}$ and $-\epsilon_{0}$ are not norms in $\mathbb{K }$, which implies that:   $$r_2(H) = \mu + r^{*} -3 =2t+3+0-3=2t.$$
 	\end{enumerate}
 	Finally,   if $n=\delta\prod_{i=1}^{t_{1}}p_{i}\prod_{j=1}^{t_{2}}q_{j}, $ with $p_{i}\equiv q_{j}\equiv3\pmod4$ and   $\left(\frac{p_{i}}{l}\right)=-\left(\frac{q_{j}}{l}\right)=-1$ for all  $i\in\{1,  \dots,  t_{1}\}$ and for all $j\in\{1,  \dots,  t_{2}\}$ with $t_{1}+t_{2}$ is odd,  then according to the two cases above,   there are $t_{1} +2t_{2}+3$  prime ideals of $k$ which ramify in  $\mathbb{K}$  and $r^{*}=0$. Thus
 	$$r_2(H)= t_{1} +2t_{2}+3+0-3 =t_{1} +2t_{2}.$$
 \end{proof}
  \subsection{\textbf{Case  $n=\delta\prod_{i=1}^tp_{i}$, with $t$  even,  and for all $i$   $p_i\equiv3\pmod4$}}
 \begin{them}\label{}
 	Let $\mathbb{K}=\mathbb{Q}(\sqrt{n\epsilon_{0}\sqrt{l}})$ be a real cyclic quartic number field,  where $l\equiv1\pmod8$ is a positive prime integer,  $n$ a  square-free  positive integer relatively prime to $l$ and  $\epsilon_{0}$ the fundamental unit of  $k=\mathbb{Q}(\sqrt{l})$.
 	Let  $n=\delta\prod_{i=1}^{i=t}p_{i}$ with $p_{i}\equiv 3\pmod4  $ for all $i\in\{1,  \dots,  t \}$ and $t$ is an even positive integer $(\delta=1$ or $2)$.
 	\begin{enumerate}[\rm1.]
 		\item If,  for all $i$,  $(\frac{p_{i}}{l})=-1$,  then  $r_2(H) =t-1+2(\delta-1)$.
 		\item If,  for all $i$,  $(\frac{p_{i}}{l})= 1$,  then  $r_2(H) =2t-2+2(\delta-1)$.
 	\end{enumerate}
 	Moreover,  for $n=\delta\prod_{i=1}^{i=t_{1}}p_{i}\prod_{j=1}^{j=t_{2}}q_{j}$ with   $(\frac{p_{i}}{l})=-(\frac{q_{j}}{l})=-1$ and $p_{i}\equiv q_{j}\equiv3\, \pmod4$ for all $i\in\{1,  \dots,  t_{1}\}$ and for all  $j\in\{1,  \dots,  t_{2}\}$ with $ t_{1}+t_{2} $ is even,  we have $r_2(H) =t_{1}+2t_{2}-2+2(\delta-1)$.
 \end{them}
 \begin{proof} For  $n=\prod_{i=1}^{i=t}p_{i}\equiv1\pmod4$, with  $p_{i}\equiv3\pmod4$ for all  $i=1,  \dots,  t$, it is easy to see that   $f_\mathbb{K}=al=\prod_{i=1}^{i=t}p_{i}l\not \equiv 0\, \pmod8$, and for $n=2\prod_{i=1}^{i=t}p_{i}$, we prove that  $f_\mathbb{K}= 2^{3}l \prod_{i=1}^{i=t} p_{i} \equiv 0 \pmod8$ .
 We proceed as in the cases above to prove,  for any even  nonzero positive integer $t$,  the existence  of real number fields $\mathbb{K}$ having $f_\mathbb{K}$ as a conductor and $k$ as a quadratic subfield.
 	\begin{enumerate}[\rm1.]
 		\item Assume $(\frac{p_{i}}{l})=-1$   for all $i$.\\
    $\bullet$ For $\delta=1$, the prime ideals of $k$ which ramify in $\mathbb{K}$ are $(\sqrt{l})$ and  $\mathfrak{p}_{i}$,  the prime ideals of $k$ above $p_i$,  thus $\mu= t + 1$. Then we have:
  $$\begin{array}{ll}
  \left( \frac{-1,d}{\sqrt{l}}\right) =\left( \frac{-1,d}{p_i}\right) =1 \, \text{and}\,
     \left( \frac{\pm\epsilon_{0},d}{p_i}\right)=-1\, \text{so just }  -1  \,\text{ is norm in} \,\mathbb{K }.\end{array}$$	
 		Then $r^{*}=1$,  from which we infer that:  $$r_2(H) = \mu + r^{*} -3 = t + 1 + 1-3=t-1.$$
 $\bullet$  For $\delta=2$,  the prime ideals of $k$ ramifying  in $\mathbb{K}$ are $(\sqrt{l})$,   $\mathfrak{2}_{i}, \;i \in \{1, 2\} $,  and  $\mathfrak{p}_{i}$,  the prime ideal of $k$ above  $p_i$, thus $\mu= t + 3$. We have:
		$$\begin{array}{ll}\left( \frac{-1,\, d}{\mathfrak{2}_{1}}\right) =\left( \frac{-1,\, d}{\mathfrak{2}_{2}}\right)=\left( \frac{-1,\, 2\prod_{i=1}^{i=t}p_{i}\epsilon_{0}\sqrt{l}}{\mathfrak{2}_{1}}\right)=(-1)^{t}=1 \;( t\; \text{is even}),\;\text{and}\, \left( \frac{\pm\epsilon_{0},d}{p_i}\right) =-1.
	\end{array}$$
		Hence $r^{*}=1$,  which implies that: $$r_2(H) = \mu + r^{*} -3 =t+ 3 +1-3=t+1.$$
 		\item Assume  $(\frac{p_{i}}{l})=1$   for all $i$.\\
    $\bullet$ For $\delta=1$,  the prime ideals of $k$ which ramify in $\mathbb{K}$ are $(\sqrt{l})$,  $\wp_{i}$ and $\bar{\wp}_{i}$   with $p_i\mathcal{O}_k=\wp_{i}\bar{\wp}_{i}$,  thus  $\mu= 2t + 1$.
 		As above,  we have:
 		$$\begin{array}{ll}\left( \frac{-1,\, d}{\mathfrak{\wp}_{i}}\right) =\left( \frac{-1,\, d}{\mathfrak{\bar{\wp}}_{i}}\right)=-1,
 		\left(\frac{\epsilon_{0},  d}{\wp_{i}}\right)\neq\left(\frac{\epsilon_{0},  d}{\bar{\wp}_{i}}\right)= \,\text{and}\,
 		\left(\frac{-\epsilon_{0},  d}{\wp_{i}}\right)\neq\left(\frac{-\epsilon_{0},  d}{\bar{\wp}_{i}}\right).
 		\end{array}$$ 	
 	So $r^{*}=0$,  from which we infer that:  $$r_2(H) = \mu + r^{*} -3 = 2t+1 + 0-3=2t-2.$$
 $\bullet$ For $\delta=2$,  then  the prime ideals of $k$ ramifying in  $\mathbb{K}$ are $(\sqrt{l})$,  $\mathfrak{2}_{i}, \;i \in \{1, 2\} $,   $\wp_{i}$   and $\bar{\wp}_{i}$,  where $p\mathcal{O}_k=\wp_{i}\bar{\wp}_{i}$  and $2\mathcal{O}_k=\mathfrak{2}_{1}\mathfrak{2}_{2}$. As above we have:
 $$\begin{array}{ll}\left( \frac{-1,\, d}{\mathfrak{\wp}_{i}}\right) =\left( \frac{-1,\, d}{\mathfrak{\bar{\wp}}_{i}}\right)=-1,
 \left(\frac{\epsilon_{0},  d}{\wp_{i}}\right)\neq\left(\frac{\epsilon_{0},  d}{\bar{\wp}_{i}}\right)\,\text{and}\,
 \left(\frac{-\epsilon_{0},  d}{\wp_{i}}\right)\neq\left(\frac{-\epsilon_{0},  d}{\bar{\wp}_{i}}\right)
 \end{array}$$ 		
		Hence $r^{*}=0$,  so: $$r_2(H) = \mu + r^{*} -3 = 2t + 3 + 0 -3=2t.$$
	\end{enumerate}
	According to previous cases,   if $n=\delta\prod_{i=1}^{t_{1}}p_{i}\prod_{j=1}^{t_{2}}q_{j}$ with $t_{1}+t_{2}$ even and  $ p_{i}\equiv q_{j}\equiv3\pmod4$, $\left(\frac{p_{i}}{l}\right)=-\left(\frac{q_{j}}{l}\right)=-1, $ for all  $i\in\{1,  \dots,  t_{1}\}$   and for all $j\in\{1,  \dots,  t_{2}\}$,   then $r^{*}=0$ and  $$r_2\left(H\right)=  t_{1} +2t_{2}-2+2(\delta-1).$$	
\end{proof}
   \subsection{\textbf{Case  $n =\delta\prod_{i=1}^{i=t}p_{i}\prod _{j=1}^{j=s}q_{j}$,   where $ p_{i}\equiv -q_{j}\equiv 1\pmod4 $ for all $ (i, j)$ and $s$  odd}}
 \begin{them}
 	Let $\mathbb{K}=\mathbb{Q}(\sqrt{n\epsilon_{0}\sqrt{l}})$ be a real cyclic quartic number field,  where $l\equiv1\pmod8$ is a positive prime integer,  $n$ a square-free positive  integer relatively prime to $l$ and  $\epsilon_{0}$ the fundamental unit of  $k=\mathbb{Q}(\sqrt{l})$. Assume  $n =\delta\prod_{i=1}^{i=t}p_{i}\prod _{j=1}^{j=s}q_{j}$  with $s$  odd and $ p_{i}\equiv- q_{j}\equiv 1 \pmod4$ are prime integers,  for all $(i,  j)\in\{1,  \dots,  t\}\times\{1,  \dots,  s\}$. Denote by $h$ the number of prime ideals of $k$ above all the  $p_{i}'s$.
 	\begin{enumerate}[\rm1.]
 		\item If,  for all $j$,   $\left(\frac{q_{j}}{l}\right)=-1$,   then  $r_2(H) =h+s$.
 		\item If,  for all $j$,  $\left(\frac{q_{j}}{l}\right)=1$,  then  $r_2(H) =h+2s$.
 	\end{enumerate}
 	Moreover,  if $\prod _{j=1}^{j=s}q_{j}=\prod_{j'=1}^{j'=s_1}q_{j'}\prod _{j=1}^{j=s_2}q_{j}$ with $\left(\frac{q_{j'}}{l}\right)=-\left(\frac{q_{j}}{l}\right)=-1$,  for all $j'=1,  \dots,  s_1$ and  $j=1,  \dots,  s_2$,   with $s_{1}+s_2$ is odd,  then  $r_2(H)=   h+s_{1}+2s_{2}$.
 \end{them}
 \begin{proof}
For the case $n =\prod_{i=1}^{i=t}p_{i}\prod _{j=1}^{j=s}q_{j}$, we have $f_\mathbb{K}=4al\not\equiv0\pmod8$ ($a=n$), and for
 $n =2\prod_{i=1}^{i=t}p_{i}\prod _{j=1}^{j=s}q_{j}$ ($a=\frac{n}{2}$), we have $f_\mathbb{K}=8al\equiv0\pmod8$. We proceed as above to prove that there exist some real cyclic quartic number fields  having  as a conductor $f_\mathbb{K}$ and as a quadratic subfield $k$.
 \begin{enumerate}[\rm1.]
 		\item If $\left(\frac{p_{i}}{l}\right)=\left(\frac{q_{j}}{l}\right)=-1$,  for all  $i\in\{1,  \dots,  t\}$ and  $j\in\{1,  \dots,  s\}$,  then the prime ideals of  $k$ which ramify in  $\mathbb{K}$ are $(\sqrt{l})$,   $\mathfrak{p}_{i}$,  $\mathfrak{q}_{j}$ the prime ideals above $p_i$, $q_j$ respectively, and $\mathfrak{2}_{i}$, $i \in \{1, 2\} $, $2\mathcal{O}_{k}=\mathfrak{2}_{1}\mathfrak{2}_{2}$. Thus $\mu= t+s+ 3$.
 		\begin{enumerate}[\rm a.]
 		 \item For the case $n =\prod_{i=1}^{i=t}p_{i}\prod _{j=1}^{j=s}q_{j}$ we have:
 		$$\begin{array}{ll}
 \left( \frac{-1,\, d}{\mathfrak{2}_{1}}\right) =\left( \frac{-1,\, d}{\mathfrak{2}_{2}}\right)=\left( \frac{-1,\, \prod_{j=1}^{j=s}q_{i}\prod_{i=1}^{i=t}p_{i}\epsilon_{0}\sqrt{l}}{\mathfrak{2}_{1}}\right)=\prod_{j=1}^{j=s}\left( \frac{-1,\, q_{j}}{\mathfrak{2}_{1}}\right)\left( \frac{-1,\, \prod_{i=1}^{i=t}p_{i}\epsilon_{0}\sqrt{l}}{\mathfrak{2}_{1}}\right)=(-1)^{s}=-1.		
 \end{array}$$	
  	\item For the case $n =2\prod_{i=1}^{i=t}p_{i}\prod _{j=1}^{j=s}q_{j}$ we have:
 	$$\begin{array}{ll}
 \left( \frac{-1,\, d}{\mathfrak{2}_{1}}\right) =\left( \frac{-1,\, d}{\mathfrak{2}_{2}}\right)=\left( \frac{-1,\, 2\prod_{j=1}^{j=s}q_{i}\prod_{i=1}^{i=t}p_{i}\epsilon_{0}\sqrt{l}}{\mathfrak{2}_{1}}\right)=\left( \frac{-1,\, 2}{\mathfrak{2}_{1}}\right)\prod_{j=1}^{j=s}\left( \frac{-1,\, q_{j}}{\mathfrak{2}_{1}}\right)\left( \frac{-1,\, \prod_{i=1}^{i=t}p_{i}\epsilon_{0}\sqrt{l}}{\mathfrak{2}_{1}}\right)=-1.\end{array}$$
  \end{enumerate}
 		Hence for the two cases, since
 		 $\left( \frac{\pm\epsilon_{0},d}{q_i}\right)
 		 =-1 $  we have: $r^{*}=0$,  so $$r_2(H) = \mu + r^{*} -3 = t+s+3 + 0-3=t+s.$$
 		\item If $\left(\frac{p_{i}}{l}\right)=-\left(\frac{q_{j}}{l}\right)=1$,  for all  $i$ and   $j$,  then the number of prime ideals in  $k$ which ramify in  $\mathbb{K}$ is  $\mu= 2t+s + 3$. Then for the two cases $n =\delta\prod_{i=1}^{i=t}p_{i}\prod _{j=1}^{j=s}q_{j}$, we have:
 		$$ \begin{array}{ll}\left( \frac{-1,\, d}{\mathfrak{2}_{1}}\right) =\left( \frac{-1,\, d}{\mathfrak{2}_{2}}\right)=\left( \frac{\pm\epsilon_{0},d}{q_i}\right)=-1\end{array}$$	
 		Hence  $r^{*}=0$,  so $$r_2(H) = \mu + r^{*} -3 = 2t+s+3+0-3=2t+s.$$
 		\item If $\left(\frac{p_{i}}{l}\right)=-\left(\frac{q_{j}}{l}\right)=-1$, for all  $i$ and   $j$,  then the number of prime ideals in  $k$ which ramify in  $\mathbb{K}$ is  $\mu= t+2s + 3$. Then for the two cases $n =\delta\prod_{i=1}^{i=t}p_{i}\prod _{j=1}^{j=s}q_{j}$, we have:
 		$$\begin{array}{ll}
 		\left( \frac{-1,d}{\rho_i}\right) =\left( \frac{-1,d}{\bar{\rho}_i}\right)=-1,
 		\left( \frac{\epsilon_{0},d}{\rho_j}\right) \neq\left( \frac{\epsilon_{0},d}{\bar{\rho}_j}\right) \,\text{and}\,\left( \frac{-\epsilon_{0},d}{\bar{\rho}_j}\right)\neq\left( \frac{-\epsilon_{0},d}{\bar{\rho}_j}\right).
 		\end{array}$$
 	 	Hence: $r^{*}=0$,  so $$r_2(H) = \mu + r^{*} -3 = t+2s+2 + 0-3=t+2s-1.$$
 		\item If $\left(\frac{p_{i}}{l}\right)=\left(\frac{q_{j}}{l}\right)=1$,  for all  $i$ and   $j$,  then the number of prime ideals of  $k$ which ramify in  $\mathbb{K}$ is  $\mu= 2t+2s + 3$. Then for the two cases $n =\delta\prod_{i=1}^{i=t}p_{i}\prod _{j=1}^{j=s}q_{j}$, as above:
 		$$\begin{array}{ll}
 		\left( \frac{-1,d}{\rho_i}\right) =\left( \frac{-1,d}{\bar{\rho}_i}\right)=-1,
 		\left( \frac{\epsilon_{0},d}{\rho_j}\right) \neq\left( \frac{\epsilon_{0},d}{\bar{\rho}_j}\right) \,\text{and}\,\left( \frac{-\epsilon_{0},d}{\bar{\rho}_j}\right)\neq\left( \frac{-\epsilon_{0},d}{\bar{\rho}_j}\right).
 		\end{array}$$
  	Hence  we have $r^{*}=0$,   so $$r_2(H) = \mu + r^{*} -3 = 2t+2s+3 + 0-3=2t+2s.$$
 	\end{enumerate}
 	In general,   if $\prod _{j=1}^{j=s}q_{j}=\prod_{j'=1}^{j'=s_1}q_{j'}\prod _{j=1}^{j=s_2}q_{j}$ with   $\left(\frac{q_{j'}}{l}\right)=-\left(\frac{q_{j}}{l}\right)=-1$,  for all $j'=1,  \dots,  s_1$ and  $j=1,  \dots,  s_2$,  with $s_{1}+s_2$ is odd,  then $r^{*}=0$ and:  $$r_2(H)= h+s_{1}+2s_{2}+3+0-3 = h+s_{1}+2s_{2}.$$
 \end{proof}
 \subsection{\textbf{Case  $n =\delta\prod_{i=1}^{i=t}p_{i}\prod _{j=1}^{j=s}q_{j}$,   where $ p_{i}\equiv -q_{j}\equiv 1\pmod4 $ for all $ (i, j)$ and $s$ is even}}
 \begin{them}
 	Let $\mathbb{K}=\mathbb{Q}(\sqrt{n\epsilon_{0}\sqrt{l}})$ be a real cyclic quartic number field,  where $l\equiv1\pmod8$ is a positive prime integer,  $n$ a square-free positive  integer relatively prime to $l$ and  $\epsilon_{0}$ the fundamental unit of  $k=\mathbb{Q}(\sqrt{l})$. Assume  $n =\prod_{i=1}^{i=t}p_{i}\prod _{j=1}^{j=s}q_{j}$  with $s$ an even positive integer and $ p_{i}\equiv- q_{j}\equiv 1 \pmod4$,  for all $(i,  j)\in\{1,  \dots,  t\}\times\{1,  \dots,  s\}$  are prime integers. Denote by $h$ the number of  prime ideals in $k$ above all the  $p_{i}'s$,   $i\in\{1,  \dots,  t\}$.
 	\begin{enumerate}[\rm1.]
 		\item If,  for all $j$,   $\left(\frac{q_{j}}{l}\right)=-1$,   then  $r_2(H) =h+s-1+2(\delta-1)$.
 		\item If,  for all $j$,  $\left(\frac{q_{j}}{l}\right)=1$,  then  $r_2(H) =h+2s-2+2(\delta-1)$.
 	\end{enumerate}
 	Moreover,  if $\prod _{j=1}^{j=s}q_{j}=\prod_{j'=1}^{j'=s_1}q_{j'}\prod _{j=1}^{j=s_2}q_{j}$ with $\left(\frac{q_{j'}}{l}\right)=-\left(\frac{q_{j}}{l}\right)=-1$,  for all $j'=1,  \dots,  s_1$ and  $j=1,  \dots,  s_2$,    $s_{1}+s_2$ is even,  then  $r_2(H)= h+s_{1}+2s_{2}-2+2(\delta-1)$.
 \end{them}
 \begin{proof}
 We proceed as in the previous cases.
 \end{proof}
 \section{\bf The Case  $l=2$}\label{3}
Let  $l=2$  and  $n$ a square-free positive integer relatively prime to $2$.  Let  $\mathbb{K}=k(\sqrt{n\epsilon_{0}\sqrt{2}})$ and $k=\mathbb{Q}(\sqrt{2})$,  where  $\epsilon_{0}$ is the fundamental unit of $k$. Since $l=2$ we have $f_\mathbb{K}= 2^{3}al=2^{3}l\prod_{i=1}^{i=t}p_{i}$. Then $f_{\mathbb{K}}\equiv0\, \pmod8$,  this implies,  by Lemma \ref{6},  that  there are as many real cyclic  fields as  imaginary ones   having as conductor $f_\mathbb{K}$ and as quadratic subfield $k$. On the other hand, the prime ideals of $k$ which ramify in $\mathbb{K}$ are $(\sqrt{2})$ and the prime ideals dividing  $n$ in $k$. Denote by $\mu$ the number of prime ideals of $k$ ramifying in $\mathbb{K}$.
\subsection{\textbf{Case  $n = 1$} }
\begin{them}\label{}
	Let $\mathbb{K}=\mathbb{Q}(\sqrt{n\epsilon_{0}\sqrt{2}})$ be a real cyclic quartic number field,  where $\epsilon_{0}$ is the fundamental unit of the quadratic subfield $k=\mathbb{Q}(\sqrt{2})$ and  $n$ a square-free positive  integer relatively prime to $2$.    If $n=1$, then $r_2\left(H\right)= 0$.
\end{them}
\begin{proof}
	In this case only $(\sqrt{2})$  ramifies in  $\mathbb{K}$,  i.e. $\mu=1$.
	$$\begin{array}{ll}
	\left( \frac{-1, d}{\sqrt{2}}\right) = \left( \frac{-1, \epsilon_{0} \sqrt{2}}{\sqrt{2}}\right) =\left( \frac{-1, \epsilon_{0} }{\sqrt{2}}\right) \left( \frac{-1, \sqrt{2}}{\sqrt{2}}\right) = \left[ \frac{-1}{\sqrt{2}}\right] = \left( \frac{-1}{2}\right) =\left( \frac{-1, 2}{2}\right) =1.\\
	\left(\frac{\epsilon_{0}, d}{\sqrt{2}}\right)=\left[\frac{\epsilon_{0}}{\sqrt{2}}\right]=\left[\frac{1+\sqrt{2}}{\sqrt{2}}\right]=\left[\frac{1}{\sqrt{2}}\right]=1.
	\end{array}$$
	Hence $r^{*}=2$,   which implies that:   $$r_2\left(H\right) = \mu + r^{*} -3 =2+1-3=0.$$
\end{proof}
\subsection{\textbf{Case  $n=\prod_{i=1}^tp_{i}$ and,  for all $i$,   $p_i\equiv1\pmod4$}}
\begin{them}\label{7}
Let $\mathbb{K}=\mathbb{Q}(\sqrt{n\epsilon_{0}\sqrt{2}})$ be a real cyclic quartic number field,  where $\epsilon_{0}$ is the fundamental unit of the quadratic subfield $k=\mathbb{Q}(\sqrt{2})$ and  $n$ a square-free positive  integer relatively prime to $2$.
 Let  $n=\prod_{i=1}^{i=t}p_{i}$ with $p_{i}\equiv 1\pmod4  $ for all $i\in\{1,  \dots,  t \}$ and $t$ is a positive integer.
	\begin{enumerate}[\rm1.]
		\item If,  for all $i$,  $(\frac{2}{p_{i}})=-1$,  then  $r_2(H) =t$.
		\item If,  for all $i$,  $(\frac{2}{p_{i}})= 1$,  then
		\begin{enumerate}[\rm a.]
			\item If $\left(\frac{2}{p_{i}}\right)_4=\left(\frac{p_i}{2}\right)_4$, for all  $i$, then $r_2(H) = 2t.$
			\item If $\left(\frac{2}{p_{i}}\right)_4\neq\left(\frac{p_i}{2}\right)_4$ for at least one $i$, then  $r_2(H) =2t-1.$
		\end{enumerate}
	\end{enumerate}
 Moreover,  if $n=\prod_{i=1}^{i=t_{1}}p_{i}\prod_{j=1}^{j=t_{2}}q_{j}$ with $\left(\frac{2}{p_{i}}\right)=-\left(\frac{2}{q_{j}}\right)=-1$ for all $ i\in\{1,  \dots,  t_{1}\}$ and for all  $j\in\{1,  \dots,  t_{2}\}$,   then:
 \begin{enumerate}[\rm a.]
 	\item If $\left(\frac{2}{q_{j}}\right)_4=\left(\frac{q_j}{2}\right)_4$  for all  $j\in\{1,  \dots,  t_{2}\}$, then $r_2(H) =t_{1}+ 2t_{2}.$
 	\item If  $\left(\frac{2}{q_{j}}\right)_4\neq\left(\frac{q_j}{2}\right)_4$  for at least one $j\in\{1,  \dots,  t_{2}\}$, then $r_2(H) =  t_{1}+ 2t_{2}-1.$
 \end{enumerate}
\end{them}
 \begin{proof}
\begin{enumerate}[\rm1.]
  \item If $(\frac{2}{p_{i}})=-1$,   for all $i=1,  \dots,  t$,   then $\mu= t + 1$. Denote by $\mathfrak{p}_{i}$ the prime ideal of $k$ above $p_i$, hence
$$\begin{array}{ll}
\left(\frac{-1,  \,   d}{\mathfrak{p}_{i}}\right)=\left[\frac{-1}{\mathfrak{p}_{i}}\right]=\left(\frac{1}{p_i}\right)=1,   \text{ for all } i=1,  \dots,   t. \\
\left(\frac{\epsilon_{0},  \,   d}{\mathfrak{p}_{i}}\right)=\left[\frac{\epsilon_{0}}{\mathfrak{p}_{i}}\right]=\left(\frac{-1}{p_{i}}\right)=1,
 \text{ for all }  i=1,  \dots,   t. \\
\left(\frac{-1,  \,   d}{\sqrt{2}}\right)=\left(\frac{\epsilon_{0},  \,   d}{\sqrt{2}}\right)=1\; \text{as above}.
\end{array}$$
 So $r^{*}=2$,  from which we infer that:  $$r_2(H) = \mu + r^{*} -3 = t + 1 + 2-3=t.$$
  \item If $(\frac{2}{p_{i}})=1$,   for all $i=1,  \dots,  t$,   then $\mu= 2t + 1$. Let $\wp_{i}$ and $\bar{\wp}_{i}$ be the prime ideals of $k$ above $p_i$, $i=1,  \dots,  t$. Hence
$$\begin{array}{ll}
\left(\frac{-1,  d}{\wp_{i}}\right)=\left(\frac{-1,  d}{\bar{\wp}_{i}}\right)=\left[\frac{-1}{\wp_{i}}\right]=\left(\frac{-1}{p_{i}}\right)=1,  \text{ for all } i=1,  \dots,  t.\\
\left(\frac{\epsilon_{0},  d}{\wp_{i}}\right)=\left(\frac{\epsilon_{0},  d}{\bar{\wp}_{i}}\right)=\left[\frac{\epsilon_{0}}{\wp_{i}}\right]=\left(\frac{2}{p_{i}}\right)_4\left(\frac{p_i}{2}\right)_4 \  (\text{see \cite[Proposition 5.8, p 160]{Lemm}}).\\
\text{Hence}\;\left(\frac{-\epsilon_{0},  d}{\wp_{i}}\right)=\left(\frac{-\epsilon_{0},  d}{\bar{\wp}_{i}}\right)=\left(\frac{2}{p_{i}}\right)_4\left(\frac{p_i}{2}\right)_4
\end{array}$$
  So
 \begin{enumerate}[\rm a.]
 	\item If $\left(\frac{2}{p_{i}}\right)_4=\left(\frac{p_i}{2}\right)_4$  for all  $i=1,  \dots,  t$, then $r^{*}=2$,  so $$r_2(H) = \mu + r^{*} -3 = 2t+1+2-3=2t.$$
 	\item If $\left(\frac{2}{p_{i}}\right)_4\neq \left(\frac{p_i}{2}\right)_4$ for at least one $i=1,  \dots,  t$, then $r^{*}=1$,  so $$r_2(H) = \mu + r^{*} -3 = 2t+1+1-3=2t-1.$$
  \end{enumerate}
  In general, if $n=\prod_{i=1}^{i=t_{1}}p_{i}\prod_{j=1}^{j=t_{2}}q_{j}$ with $\left(\frac{2}{p_{i}}\right)=-\left(\frac{2}{q_{j}}\right)=-1$ for all $ i\in\{1,  \dots,  t_{1}\}$ and for all  $j\in\{1,  \dots,  t_{2}\}$,   then according to the two cases above,
  \begin{enumerate}[\rm a.]
  	\item If $\left(\frac{2}{q_{j}}\right)_4=\left(\frac{q_j}{2}\right)_4$ for all $j=1,  \dots,  t_{2}$, then $r^{*}=2$,  so $$r_2(H) = \mu + r^{*} -3 =t_{1}+ 2t_{2}+1+2-3=t_{1}+ 2t_{2}.$$
  	\item If $\left(\frac{2}{q_{j}}\right)_4\neq\left(\frac{q_j}{2}\right)_4$ for at least one $j=1,  \dots,  t_{2}$, then $r^{*}=1$,  so $$r_2(H) = \mu + r^{*} -3 = t_{1}+ 2t_{2}+1+1-3=t_{1}+ 2t_{2}-1.$$
  \end{enumerate}
\end{enumerate}
\end{proof}
\subsection{\textbf{Case  $n = \prod_{i=1}^{i=t}p_{i}$ with  $p_i\equiv3\pmod4$ for all $i$}}
\begin{them}\label{}
	Let $\mathbb{K}=\mathbb{Q}(\sqrt{n\epsilon_{0}\sqrt{2}})$ be a real cyclic quartic number field,  where $\epsilon_{0}$ is the fundamental unit of the quadratic subfield $k=\mathbb{Q}(\sqrt{2})$ and  $n$ a square-free positive  integer relatively prime to $2$. Assume  $n = \prod_{i=1}^{i=t}p_{i}$,   $p_{i}\equiv3\pmod4$ for all  $i=1,  \dots,  t$ and $t$ is a positive integer.
\begin{enumerate}[\rm1.]
	\item If,  for all $i$,   $\left(\frac{2}{p_{i}}\right)=-1$,   then  $r_2\left(H\right) =t-1$.
	\item If,  for all $i$,   $\left(\frac{2}{p_{i}}\right)= 1$,   then  $r_2\left(H\right) =2t-2$.
\end{enumerate}	
Moreover,   if $n=\prod_{i=1}^{t_{1}}p_{i}\prod_{j=1}^{t_{2}}q_{j}$,  where  $ p_{i}\equiv q_{j}\equiv3\pmod4$ and  $\left(\frac{2}{p_{i}}\right)=-\left(\frac{2}{q_{j}}\right)=-1, $ for all  $i\in\{1,  \dots,  t_{1}\}$,   and for all $j\in\{1,  \dots,  t_{2}\}$,   then  $r_2\left(H\right)=  t_{1} +2t_{2}-2$.
\end{them}
\begin{proof}
\begin{enumerate}[\rm1.]
\item If $\left(\frac{2}{p_{i}}\right)=-1$, for all  $i\in\{1,  \dots,  t\}$,  then $\mu=t+1$.  For $\mathfrak{p}_{i}$  the prime ideal of $k$ above $p_i$   we have
$$\begin{array}{ll}
\left(\frac{-1, \,  d}{\mathfrak{p}_{i}}\right)=\left[\frac{-1}{\mathfrak{p}_{i}}\right]= \left(\frac{1}{p_{i}}\right)=1 \text{ and }
\left(\frac{\epsilon_{0}, \,  d}{\mathfrak{p}_{i}}\right)=\left[\frac{\epsilon_{0}}{\mathfrak{p}_{i}}\right]=\left(\frac{-1}{p_{i}}\right)= -1,\;
\text{hence}\, \left(\frac{-\epsilon_{0}, \,  d}{\mathfrak{p}_{i}}\right)=-1.
\end{array}$$
Since $\left(\frac{-1,  \,   d}{(\sqrt{2})}\right)=1$,  just the unit $-1$ is norm in $\mathbb{K}$.
Hence $r^{*}=1$,   which implies that:   $$r_2\left(H\right) = \mu + r^{*} -3 =t+1+1-3=t-1.$$
\item If $\left(\frac{2}{p_{i}}\right)=1$ for all  $i\in\{1,  \dots,  t\}$,  then  the prime ideals of  $k$ which ramify in  $\mathbb{K}$ are
$(\sqrt{2})$,   $\wp_{i}$ and $\bar{\wp}_{i}$,   where $p_i\mathcal{O}_k=\wp_{i}\bar{\wp}_{i}$,  $i=1,  \dots,  t$. Hence for all $i\in\{1,  \dots,  t\}$,  we have:
$$\begin{array}{ll}\left(\frac{-1, \,  d}{\wp_{i}}\right)=\left(\frac{-1,  d}{\bar{\wp}_{i}}\right)=\left[\frac{-1}{\wp_{i}}\right]=\left(\frac{-1}{p_{i}}\right)=-1.\end{array}$$
On the other hand, as $p_i\equiv3\pmod4$, so $\left(\frac{\epsilon_{0}, d}{\wp_{i}}\right)=\left[\frac{\epsilon_{0}}{\wp_{i}}\right]=-\left[\frac{\epsilon_{0}}{\bar{\wp_{i}}}\right]=
-\left(\frac{\epsilon_{0}, d}{\bar{\wp_{i}}}\right)$ (\text{see \cite[page 160]{Lemm}}). Thus  $\left(\frac{\epsilon_{0}, d}{\wp_{i}}\right)=\pm1$ and $\left(\frac{\epsilon_{0}, d}{\bar{\wp_{i}}}\right)=\mp1$, wich implies $\left(\frac{-\epsilon_{0}, d}{\wp_{i}}\right)=\mp1$ and $\left(\frac{-\epsilon_{0}, d}{\bar{\wp_{i}}}\right)=\pm1$.\\
Hence $r^{*}=0$,   and we infer that:   $$r_2(H) = \mu + r^{*} -3 =2t+1+0-3=2t-2.$$
\end{enumerate}
Finally,   if $n=\prod_{i=1}^{t_{1}}p_{i}\prod_{j=1}^{t_{2}}q_{j}, $ with $p_{i}\equiv q_{j}\equiv3\pmod4$ and   $\left(\frac{2}{p_{i}}\right)=-\left(\frac{2}{p_{i}}\right)=-1$ for all  $i\in\{1,  \dots,  t_{1}\}$ and for all $j\in\{1,  \dots,  t_{2}\}$,  then according to the two cases above,   there are $t_{1} +2t_{2}+1$  prime ideals of $k$ which ramifies in  $\mathbb{K}$  and $r^{*}=0$. Thus
$$r_2(H)= t_{1} +2t_{2}+1+0-3 =t_{1} +2t_{2}-2.$$
\end{proof}
\subsection{\textbf{Case  $n =\prod_{i=1}^{i=t}p_{i}\prod _{j=1}^{j=s}q_{j}$,   $p_{i}\equiv-q_{j}\equiv1\pmod4$  $\forall (i,  j)$ }}
\begin{them}
Let $\mathbb{K}=\mathbb{Q}(\sqrt{n\epsilon_{0}\sqrt{2}})$ be a real cyclic quartic number field,  where  $n$ a square-free positive  integer relatively prime to $2$ and  $\epsilon_{0}$ the fundamental unit of  $k=\mathbb{Q}(\sqrt{2})$. Assume $n =\prod_{i=1}^{i=t}p_{i}\prod _{j=1}^{j=s}q_{j}$ with  $p_{i}\equiv -q_{j}\equiv 1\pmod4$ for all  $(i,  j)\in\{1,  \dots,  t\}\times\{1,  \dots,  s\}$. Denote by $h$ the number of prime ideals of $k$ dividing all the  $p_{i}'s$,   $i\in\{1,  \dots,  t\}$.
	\begin{enumerate}[\rm1.]
	\item If,  for all $j$,   $\left(\frac{2}{q_{j}}\right)=-1$,  then  $r_2(H) =h+s-1$.
	\item If,  for all $j$,  $\left(\frac{2}{q_{j}}\right)=1$,  then  $r_2(H) =h+2s-2$.
    \end{enumerate}
Moreover,    if  $\prod _{j=1}^{j=s}q_{j}=\prod _{i=1}^{i=s_{1}}q_{i}\prod _{j=1}^{j=s_{2}}q_{j}$ with    $q_{i}\equiv q_{j}\equiv3\,  \pmod4$ and $\left(\frac{2}{q_{j}}\right)=-\left(\frac{2}{q_{j}}\right)=-1,   i\in\{1,  \dots,  s_{1}\},   j\in\{1,  \dots,  s_{2}\}$,  then  $r_2(H)=h+s_{1}+2s_{2}-2$.
\end{them}
\begin{proof}
\begin{enumerate}[\rm1.]
	\item  If $\left(\frac{p_{i}}{l}\right)=\left(\frac{q_{j}}{l}\right)=-1$,  for all $i=1,  \dots,  t$ and for all $j=1,  \dots,  s$,  then   the prime ideals of  $k$ which ramify in $\mathbb{K}$ are $(\sqrt{2})$,  $\mathfrak{ p}_{i}$ and $\mathfrak{q}_{j}$,   the prime ideals of $k$ above ${ p}_{i}$ and ${q}_{j}$ respectively,  i.e. $\mu= t+s+1$.	As above we have:
	$$\begin{array}{ll}
 \left(\frac{-1,d}{\sqrt{2}}\right) =\left(\frac{-1,d}{\mathfrak{ p}_i} \right)=\left( \frac{-1,d}{\mathfrak{q}_j} \right) =1\; \text{and}\; \left( \frac{\pm\epsilon_{0},d}{\mathfrak{q}_j}\right) =-1.\end{array}$$
 thus  $r^{*}=1$,  so $$r_2(H) = \mu + r^{*} -3 = t+s+1 + 1-3=t+s-1.$$
\item If $\left(\frac{p_{i}}{l}\right)=-\left(\frac{q_{j}}{l}\right)=1$,  for all $i=1,  \dots,  t$ and for all $j=1,  \dots,  s$,  then  the prime ideals of  $k$ which ramifies in  $\mathbb{K}$ are $(\sqrt{2})$,   $\wp_{i}$,   $\bar{\wp}_{i}$ and $\mathfrak{q}_{j}$,   where  $p\mathcal{O}_k=\wp_{i}\bar{\wp}_{i}$ and $\mathfrak{q}_{j}$ is the prime ideal of $k$ above  ${q}_{j}$,  i.e. $\mu= 2t+s+1$. As above we have:
$$\begin{array}{ll} \left(\frac{-1,d}{\sqrt{2}}\right) =\left(\frac{-1,d}{\mathfrak{ \wp}_i} \right)=\left(\frac{-1,d}{\mathfrak{ \bar{\wp}}_i} \right)=\left( \frac{-1,d}{\mathfrak{q}_j} \right) =1\; \text{and}\; \left( \frac{\pm\epsilon_{0},d}{\mathfrak{q}_j}\right) =-1.\end{array}$$
 thus $r^{*}=1$,  so $$r_2(H) = \mu + r^{*} -3 = 2t+s+1 + 1-3=2t+s-1.$$

  \item If $\left(\frac{p_{i}}{l}\right)=-\left(\frac{q_{j}}{l}\right)=-1$,  for all $i=1,  \dots,  t$ and  for all $j=1,  \dots,  s$,  then  the prime ideals of  $k$ which ramify in  $\mathbb{K}$ are $(\sqrt{2})$,   $\mathfrak{p}_{i}$,   $\rho_{j}$,   and $\bar{\rho}_{j}$,  where   $q_j\mathcal{O}_k=\rho_{j}\bar{\rho}_{j}$ and $\mathfrak{p}_{j}$ is the prime ideal of $k$ above  ${p}_{j}$,  i.e. $\mu= t+2s+1$.  As above we have:
  $$\begin{array}{ll} \left(\frac{-1,d}{\rho_j} \right)=\left( \frac{-1,d}{\bar{\rho}_j} \right) =-1\; \text{and}\; \left( \frac{\pm\epsilon_{0},d}{\rho_j}\right) \neq\left( \frac{\pm\epsilon_{0},d}{\bar{\rho}_j}\right)\end{array}$$
 thus $r^{*}=0$,  so $$r_2(H) = \mu + r^{*} -3 = t+2s+1 + 0-3=t+2s-2.$$

 \item If $\left(\frac{p_{i}}{l}\right)=\left(\frac{q_{j}}{l}\right)=1$,  for all $i=1,  \dots,  t$ and for all $j=1,  \dots,  s$,  then   the prime ideals of  $k$ which ramifies in $\mathbb{K}$ are $ (\sqrt{2})$,   $\wp_{i}$,   $\bar{\wp}_{i}$,   $\rho_{j}$ and $\bar{\rho}_{j}$,   where $p_i\mathcal{O}_k=\wp_{i}\bar{\wp}_{i}$ and $q_j\mathcal{O}_k=\rho_{j}\bar{\rho}_{j}$,  i.e. $\mu= 2t+2s+1$. As above we have:
 $$\begin{array}{ll} \left(\frac{-1,d}{\rho_j} \right)=\left( \frac{-1,d}{\bar{\rho}_j} \right) =-1\; \text{and}\; \left( \frac{\pm\epsilon_{0},d}{\rho_j}\right) \neq\left( \frac{\pm\epsilon_{0},d}{\bar{\rho}_j}\right)\end{array}$$
 thus $r^{*}=0$,  so  $$r_2(H) = \mu + r^{*} -3 = 2t+2s+1 + 0-3=2t+2s-2.$$
\end{enumerate}
In general,    if  $\prod _{j=1}^{j=s}q_{j}=\prod _{i=1}^{i=s_{1}}q_{i}\prod _{j=1}^{j=s_{2}}q_{j}$,  with  $q_{i}\equiv q_{j}\equiv3\,  \pmod4$ and $\left(\frac{q_{i}}{l}\right)=-\left(\frac{q_{j}}{l}\right)=-1,   i\in\{1,  \dots,  s_{1}\},   j\in\{1,  \dots,  s_{2}\}$,  then $r^{*}=0$ and $$r_2(H)= h+s_{1}+2s_{2} + 1 + 0 -3=h+s_{1}+2s_{2}-2, $$ where  $h$ is always the number of  prime divisors of all the $p_{i}$'s,     $p_{i}\equiv1\,  \pmod4,   i\in\{1,  \dots,  t\}$ in  $k$.
\end{proof}
\section{Applications}\label{4}
In this section,  we will determine the integers  $n$ such that  $r_2(H)$,  the rank of the $2$-class group  $H$ of $\mathbb{K}=\mathbb{Q}(\sqrt{n\epsilon_{0}\sqrt{l}})$,   is equal to $0$,  $1$,  $2$ or $3$. For this we adopt the following notations: $p$ and $p_i$ (resp. $q$ and $q_i$),  $i\in\NN^*$,  are prime integers  congruent to $1$ (resp. $3$) modulo $4$. $\delta=1$ or $2$. The following theorems are simple deductions from the results of  previous subsections.  For all the examples below,  we use PARI/GP calculator version 2.11.2 (64bit), April 28, 2019.
\subsection{Case $l\equiv 1\pmod8$}	
\begin{them}\label{16}
	Let $\mathbb{K}=\mathbb{Q}(\sqrt{n\epsilon_{0}\sqrt{l}})$ be a real cyclic quartic number field,  where $l\equiv1\pmod8$ is a positive prime integer,  $n$ a square-free positive  integer relatively prime to $l$ and  $\epsilon_{0}$ the fundamental unit of  $k=\mathbb{Q}(\sqrt{l})$. The class number of $\mathbb{K}$ is odd if,  and only if $n=1$.
\end{them}
\begin{exams}
	For $n=1$ and $l=257\equiv1 \pmod8$, we have the class number of the class group $H$ of $\mathbb{K}=\mathbb{Q}(\sqrt{\epsilon_{0}\sqrt{l}})$ is $3$.
\end{exams}
\begin{them}\label{9}
	Let $\mathbb{K}=\mathbb{Q}(\sqrt{n\epsilon_{0}\sqrt{l}})$ be a real cyclic quartic number field,  where $l\equiv1\pmod8$ is a prime,  $n$ a square-free positive  integer relatively prime to $l$ and  $\epsilon_{0}$ the fundamental unit of  $k=\mathbb{Q}(\sqrt{l})$.
	$H$ is cyclic  if,  and only if one of the following assertions holds:
	\begin{enumerate}[\rm1.]
		\item $n= p$ and
		\begin{enumerate}[\rm i.]
			\item either $(\frac{p}{l})=-1$,
			\item or $(\frac{p}{l})= 1$ and $\left(\frac{p}{l}\right)_{_{4}}\neq\left(\frac{l}{p}\right)_{_{4}} $.
		\end{enumerate}
		\item $n= 2$ and $\left( \frac{2}{l}\right) _{4}\neq (-1)^{\frac{l-1}{8}}$.
		\item $n= \delta q$ and $(\frac{q}{l})=-1$.
		\item $n=q_{1}q_{2}$ and $(\frac{q_{1}}{l})=-1$ or $(\frac{q_{2}}{l})=-1$. 	
	\end{enumerate}
\end{them}
\begin{exams}~\
	\begin{enumerate}[\rm1.]
		\item	For $n=p=89\equiv 1\pmod4$ and $l=41\equiv1 \pmod8$,  $(\frac{p}{l})=-1$ and $H$ is cyclic of order $2$. 	
		 For $n=p=97\equiv1\pmod4$ and $l=89\equiv1\pmod8$,   $\left(\frac{p}{l}\right)_{_{4}}=-\left(\frac{l}{p}\right)_{_{4}}=1 $ and $H$ is cyclic of order $2$.
		\item For $n=2$ and  and $l=1913\equiv1\pmod8$,  $\left( \frac{2}{l}\right) _{4}= -(-1)^{\frac{l-1}{8}}=1$ and $H$ is cyclic of order $2$.
		\item For $n=q=83\equiv3\pmod4$ and $l=137\equiv1\pmod8$,  $(\frac{q}{l})=-1$ and $H$ is cyclic of order $2$.
		 For $n=2q=2.83\equiv2\pmod4$ and $l=97\equiv1\pmod8$,  $(\frac{q}{l})=-1$ and $H$ is cyclic of order $2$.
		\item For $n=q_1q_2=71.83\equiv1\pmod4$ and $l=97\equiv1\pmod8$,  $(\frac{q_{1}}{l})=(\frac{q_{2}}{l})=-1$ and $H$ is cyclic of order $2$.
		 For $n=q_1q_2=79.83\equiv1\pmod4$ and $l=41\equiv1\pmod8$,  $(\frac{q_{1}}{l})=-(\frac{q_{2}}{l})=-1$ and $H$ is cyclic of order $2$.
	\end{enumerate}
\end{exams}
\begin{them}\label{10}
	Let $\mathbb{K}=\mathbb{Q}(\sqrt{n\epsilon_{0}\sqrt{l}})$ be a real cyclic quartic number field,  where $l\equiv1\pmod8$ is a prime,  $n$ a square-free positive  integer relatively prime to $l$ and  $\epsilon_{0}$ the fundamental unit of  $k=\mathbb{Q}(\sqrt{l})$.
	The rank $r_2(H)$ equals $2$  if,  and only if $n$  takes one of the following forms.
	\begin{enumerate}[\rm1.]
		\item $n= p_{1}p_{2}$ and
		\begin{enumerate}[\rm i.]
			\item either $(\frac{p_{i}}{l})= -1$  for all $i\in\{1, 2\}$
			\item or $(\frac{p_{i}}{l})=-(\frac{p_{j}}{l})=-1$ and $\left(\frac{p_j}{l}\right)_{_{4}}\neq\left(\frac{l}{p_j}\right)_{_{4}}$, $i\neq j\in\{1, 2\}$.
		\end{enumerate}
		\item $n=p$, $(\frac{p}{l})=1$ and $\left(\frac{p}{l}\right)_{_{4}}=\left(\frac{l}{p}\right)_{_{4}} $.
		\item $n=2p$, $(\frac{p}{l})=-1$  and $\left( \frac{2}{l}\right) _{4}\neq (-1)^{\frac{l-1}{8}}$.
		\item $n=2$ and  $\left( \frac{2}{l}\right) _{4}= (-1)^{\frac{l-1}{8}}$.
		\item $n=\delta q$  and  $(\frac{q}{l})=1$.
		\item $n=q_1q_2$  and $(\frac{q_1}{l})=(\frac{q_2}{l})=1$.
		\item $n=\delta pq$ and   $(\frac{p}{l})=(\frac{q}{l})=-1$.
		\item $n=pq_{1}q_{2}$, $(\frac{p}{l})=-1$ and  $(\frac{q_1}{l})=-1$ or  $(\frac{q_2}{l})=-1$.				
	\end{enumerate}
\end{them}
\begin{exams}~\
	\begin{enumerate}[\rm1.]
		\item	For $n=p_1p_2=89.97\equiv1\pmod4$ and $l=41\equiv1\pmod8$, we have $(\frac{p_{1}}{l})=(\frac{p_{2}}{l})= -1$ and $H$ is  of type $(2, 2)$. For $n=p_1p_2=97.89\equiv1\pmod4$ and $l=17\equiv1\pmod8$, we have $(\frac{p_{1}}{l})=-(\frac{p_{2}}{l})=-1$ and $\left(\frac{p_2}{l}\right)_{_{4}}=-\left(\frac{l}{p_2}\right)_{_{4}}=1 $, $H$ is  of type $(2, 2)$.
		\item $n=p=613\equiv1\pmod4$ and $l=17\equiv1\pmod8$, we have $(\frac{p}{l})=1$,  $\left(\frac{p}{l}\right)_{_{4}}=\left(\frac{l}{p}\right)_{_{4}}=1 $ and $H$ is  of type $(4,4) $.
		\item $n=2p=2.1994\equiv2\pmod4$ and $l=1753\equiv1\pmod8$, we have $(\frac{p}{l})=-1$,  $\left( \frac{2}{l}\right) _{4}=-(-1)^{\frac{l-1}{8}}=1$ and $H$ is  of type $(2,2) $.
		\item For $n=2$ and  and $l=1889\equiv1\pmod8$, we have $\left( \frac{2}{l}\right) _{4}= (-1)^{\frac{l-1}{8}}=1$ and $H$ is  of type $(2,4)$.
		\item For $n=q=79\equiv3\pmod4$ and $l=97\equiv1\pmod8$, we have $(\frac{q}{l})=1$ and $H$ is  of type $(2,2) $.
		 For $n=2q=2.71\equiv2\pmod4$ and $l=73\equiv1\pmod8$, we have $(\frac{q}{l})=1$ and $H$ is  of type $(2,4) $.
		\item For $n=q_1q_2=47.67\equiv1\pmod4$ and $l=17\equiv1\pmod8$, we have $(\frac{q_{1}}{l})=(\frac{q_{2}}{l})=1$ and $H$ is  of type $(2,4) $.
		\item For $n=pq=73.79\equiv3\pmod4$ and $l=17\equiv1\pmod8$, we have $(\frac{p}{l})=(\frac{q}{l})=-1$ and $H$ is  of type $(2,2) $.
		 For $n=2pq=2.41.79\equiv2\pmod4$ and $l=17\equiv1\pmod8$, we have $(\frac{p}{l})=(\frac{q}{l})=-1$ and $H$ is  of type $(2,2) $.
		\item For $n=pq_{1}q_{2}=97.71.79\equiv1\pmod4$ and $l=41\equiv1\pmod8$, we have $(\frac{p}{l})=(\frac{q_{1}}{l})=(\frac{q_{2}}{l})=-1$ and $H$ is of type $(2,2) $.
		For $n=pq_{1}q_{2}=97.79.83\equiv1\pmod4$ and $l=41\equiv1\pmod8$, we have $(\frac{p}{l})=(\frac{q_{1}}{l})=-(\frac{q_{2}}{l})=-1$ and $H$ is of type $(2,2) $.			
	\end{enumerate}
\end{exams}
\begin{them}\label{11}
	Let $\mathbb{K}=\mathbb{Q}(\sqrt{n\epsilon_{0}\sqrt{l}})$ be a real cyclic quartic number field,  where $l\equiv1\pmod8$ is a prime,  $n$ a square-free positive  integer relatively prime to $l$ and  $\epsilon_{0}$ the fundamental unit of  $k=\mathbb{Q}(\sqrt{l})$.
	The rank $r_2(H)$ equals $3$  if  and only if $n$  takes one of the following forms.
	\begin{enumerate}[\rm1.]
\item $n=2p$ and one of the following cases holds:
		\begin{enumerate}[\rm i.]
			\item    $(\frac{p}{l})=-1$ and $\left( \frac{2}{l}\right) _{4}= (-1)^{\frac{l-1}{8}}$ ,
			\item   $(\frac{p}{l})=1$ and $\left( \frac{2}{l}\right)_{4}\neq (-1)^{\frac{l-1}{8}}$ or $\left(\frac{p}{l}\right)_{_{4}}\neq \left(\frac{l}{p}\right)_{_{4}} $.
		\end{enumerate}
		\item $n= p_1p_2$ and
		\begin{enumerate}[\rm i.]
			\item either  $(\frac{p_i}{l})=1, \; \text{for all}\; i\in\{1,  2\}$ and $\left(\frac{p_{i}}{l}\right)_{_{4}}\neq\left(\frac{l}{p_{i}}\right)_{_{4}} $ for at least one $i\in\{1,  2\}$ .
			\item or $(\frac{p_1}{l})=-(\frac{p_2}{l})=-1$  and $\left(\frac{p_{2}}{l}\right)_{_{4}}=\left(\frac{l}{p_{2}}\right)_{_{4}} $.
		\end{enumerate}
		\item $n=2p_1p_2$, $(\frac{p_i}{l})=-1$  for all $i\in\{1, 2\}$ and $\left( \frac{2}{l}\right) _{4}\neq (-1)^{\frac{l-1}{8}}$.
		\item $n=p_1p_2p_3$ and
		\begin{enumerate}[\rm i.]
			\item either $(\frac{p_i}{l})=-1$  for all $i\in\{1, 2, 3\}$,
			\item or $(\frac{p_1}{l})=(\frac{p_2}{l})=-(\frac{p_3}{l})=-1$ and $\left(\frac{p_{3}}{l}\right)_{_{4}}\neq\left(\frac{l}{p_{3}}\right)_{_{4}} $.
		\end{enumerate}
\item $n= \delta q_1q_2q_3$ and $(\frac{q_i}{l})=-1$  for all $i\in\{1,  2,  3\}$
		\item $n=q_1q_2q_3q_4$ and there exist at most one of symbols $(\frac{q_i}{l})$ for $i\in\{1, 2, 3, 4\}$ equal $1$.
		\item $n=2q_1q_2$ and $(\frac{q_i}{l})=-1$  for at least one  $i\in\{1,  2\}$.
		\item $n=\delta p_1p_2q$,   $(\frac{p_i}{l})=-1$ for all $i\in\{1,  2\}$ and $(\frac{q}{l})=-1$.
		\item $n=\delta pq$ and $(\frac{p}{l})\neq(\frac{q}{l})$.
		\item $n=p_1p_2q_1q_2$,  $(\frac{p_i}{l})=-1$  for all $i\in\{1, 2\}$ and  $(\frac{q_i}{l})=1$ for at most one  $i\in\{1,  2\}$.
		\item $n= pq_1q_2$ and one of the following cases holds:
		\begin{enumerate}[\rm i.]
			\item  $(\frac{p}{l})=1$ and at most one of the  symbols $(\frac{q_i}{l})$,  $i\in\{1,  2\}$,   is $1$.
		\item   $(\frac{p}{l})=-1$ and  $(\frac{q_i}{l})=1$  for all $i\in\{1,  2\}$.
         \end{enumerate}	
	\end{enumerate}
\end{them}
\begin{exams}~\
	\begin{enumerate}[\rm1.]
		\item For $n=p_{1}p_{2}p_{3}=37.41.61\equiv1\pmod4$ and $l=89\equiv1\pmod8$, we have $(\frac{p_{1}}{l})=(\frac{p_{2}}{l})=(\frac{p_{3}}{l})=-1$ and $H$ is of type $(2,2,2)$.
		 For $n=p_{1}p_{2}p_{3}=89.97.73\equiv1\pmod4$ and $l=41\equiv1\pmod8$, we have $(\frac{p_{1}}{l})=(\frac{p_{2}}{l})=-(\frac{p_{3}}{l})=-1$ and $\left(\frac{p_{3}}{l}\right)_{_{4}}=-\left(\frac{l}{p_{3}}\right)_{_{4}}=-1 $, $H$ is of type $(2,2,2)$
		\item	For $n=p_{1}p_{2}=89.97\equiv1\pmod4$ and $l=73\equiv1\pmod8$, we have $(\frac{p_{1}}{l})=(\frac{p_{2}}{l})=1$ and $\left(\frac{p_{2}}{l}\right)_{_{4}}=-\left(\frac{l}{p_{2}}\right)_{_{4}}=-1 $,  $H$ is of type $(2,4,4)$. 
		For $n=p_{1}p_{2}=61.73\equiv1\pmod4$ and $l=89\equiv1\pmod8$, we have $(\frac{p_{1}}{l})=-(\frac{p_{2}}{l})=-1$ and $\left(\frac{p_{2}}{l}\right)_{_{4}}=\left(\frac{l}{p_{2}}\right)_{_{4}}=1 $,  $H$ is of type $(2,2,4)$
		\item For $n=2p=2.97\equiv2\pmod4$ and $l=41\equiv1\pmod8$, we have $(\frac{p}{l})=-1$ and $\left( \frac{2}{l}\right) _{4}= (-1)^{\frac{l-1}{8}}=-1$,  $H$ is of type $(2,2,2)$.
		For $n=2p=2.97\equiv2\pmod4$ and $l=89\equiv1\pmod8$, we have $(\frac{p}{l})=1$, $\left( \frac{2}{l}\right) _{4}=- (-1)^{\frac{l-1}{8}}=1$ and $\left(\frac{p}{l}\right)_{_{4}}=- \left(\frac{l}{p}\right)_{_{4}}=1 $,  $H$ is of type $(2,4,4)$
		\item For $n=2p_{1}p_{2}=2.73.97\equiv1\pmod4$ and $l=17\equiv1\pmod8$, we have $(\frac{p_{1}}{l})=(\frac{p_{2}}{l})=-1$ and $\left( \frac{2}{l}\right) _{4}=- (-1)^{\frac{l-1}{8}}=-1$,  $H$ is of type $(2,2,2)$
		\item For $n=q_{1}q_{2}q_{3}=67.71.83\equiv3\pmod4$ and $l=97\equiv1\pmod8$, we have $(\frac{q_{1}}{l})=(\frac{q_{2}}{l})=(\frac{q_{3}}{l})=-1$ and $H$ is of type $(2,2,2)$. 
		For $n=2q_{1}q_{2}q_{3}=59.67.83\equiv2\pmod4$ and $l=97\equiv1\pmod8$, we have $(\frac{q_{1}}{l})=(\frac{q_{2}}{l})=(\frac{q_{3}}{l})=-1$ and $H$ is of type $(2,2,2)$
		\item For $n=q_{1}q_{2}q_{3}q_{4}=11.23.31.7\equiv1\pmod4$ and $l=17\equiv1\pmod8$, we have $(\frac{q_{1}}{l})=(\frac{q_{2}}{l})=(\frac{q_{3}}{l})=(\frac{q_{3}}{l})=-1$ and $H$ is of type $(2,2,2)$.
		\item For $n=2q_{1}q_{2}=71.83\equiv2\pmod4$ and $l=97\equiv1\pmod8$, we have $(\frac{q_{1}}{l})=(\frac{q_{2}}{l})=-1$ and $H$ is of type $(2,2,4)$.
		For $n=2q_{1}q_{2}=71.83\equiv2\pmod4$ and $l=41\equiv1\pmod8$, we have $(\frac{q_{1}}{l})=-(\frac{q_{2}}{l})=-1$ and $H$ is of type $(2,2,2)$
		\item For $n=p_{1}p_{2}q=89.97.79\equiv3\pmod4$ and $l=41\equiv1\pmod8$, we have $(\frac{p_{1}}{l})=(\frac{p_{2}}{l})=(\frac{q}{l})=-1$ and $H$ is of type $(2,2,2)$.
		For $n=2p_{1}p_{2}q=2.89.97.79\equiv3\pmod4$ and $l=41\equiv1\pmod8$, we have $(\frac{p_{1}}{l})=(\frac{p_{2}}{l})=(\frac{q}{l})=-1$ and $H$ is of type $(2,2,2)$
		\item For $n=pq=61.47\equiv3\pmod4$ and $l=17\equiv1\pmod8$, we have $(\frac{p}{l})=-(\frac{q}{l})=-1$ and $H$ is of type $(2,4,8)$. 
		For $n=2pq=2.97.47\equiv2\pmod4$ and $l=17\equiv1\pmod8$, we have $(\frac{p}{l})=-(\frac{q}{l})=1$ and $H$ is of type $(2,2,2)$
		\item For $n=p_{1}p_{2}q_{1}q_{2}=61.73.71.79\equiv1\pmod4$ and $l=17\equiv1\pmod8$, we have $(\frac{p_{1}}{l})=(\frac{p_{2}}{l})=(\frac{q_{1}}{l})=(\frac{q_{2}}{l})=-1$ and $H$ is of type $(2,2,2)$.
		For $n=p_{1}p_{2}q_{1}q_{2}=61.73.67.79\equiv1\pmod4$ and $l=17\equiv1\pmod8$, we have $(\frac{p_{1}}{l})=(\frac{p_{2}}{l})=-(\frac{q_{1}}{l})=(\frac{q_{2}}{l})=-1$ and $H$ is of type $(2,2,2)$
		\item For $n=pq_{1}q_{2}=97.83.79\equiv1\pmod4$ and $l=89\equiv1\pmod8$, we have $(\frac{p}{l})=-(\frac{q_{1}}{l})=(\frac{q_{2}}{l})=1$ and $H$ is of type $(2,2,2)$. 
	\end{enumerate}
\end{exams}
 \subsection{Case $l=2$}
 \begin{them}\label{12}
	Let $\mathbb{K}=\mathbb{Q}(\sqrt{n\epsilon_{0}\sqrt{2}})$ be a real cyclic quartic number field,   $n$ an odd square-free positive  integer  and  $\epsilon_{0}$ the fundamental unit of  $k=\mathbb{Q}(\sqrt{2})$. The class number of $\mathbb{K}$ is odd if,  and only if one of the following assertions holds:
	\begin{enumerate}[\rm1.]
		\item $n=1$.
		\item $n$ is a prime integer congruent to $3\pmod4$.	
	\end{enumerate}
\end{them}
\begin{exams}   For $n=q=59\equiv3\pmod4$, we have $\left( \frac{2}{q}\right) =-1$, $H$ has order $5$.
		 For $n=q=631\equiv3\pmod4$, we have $\left( \frac{2}{q}\right) =1$, $H$ has order $5$. 
\end{exams}
\begin{them}\label{13}
	Let $\mathbb{K}=\mathbb{Q}(\sqrt{n\epsilon_{0}\sqrt{2}})$ be a real cyclic quartic number field,  $n$ an odd square-free positive  integer  and  $\epsilon_{0}$ the fundamental unit of  $k=\mathbb{Q}(\sqrt{2})$.	$H$ is cyclic  if,  and only if one of the following assertions holds:
	\begin{enumerate}[\rm1.]
		\item $n=p$ and
		\begin{enumerate}[\rm a.]
			\item either $(\frac{2}{p})=-1$.
			\item or $(\frac{2}{p})=1$ and $ \left( \frac{2}{p}\right)_{4} \neq\left( \frac{p}{2}\right)_{4} $
		\end{enumerate}
		\item $n=q_{1}q_{2}$ and  $(\frac{2}{q_1})=-1$ or  $(\frac{2}{q_2})=-1$.
		\item $n=pq$ and $(\frac{2}{p})=-1$ .	
	\end{enumerate}
\end{them}
\begin{exams}~\
	\begin{enumerate}[\rm1.]
		\item For $n=p=61\equiv1\pmod4$, we have $(\frac{2}{p})=-1$ and  $H$ is cyclic of order $2$. 
		For $n=p=89\equiv1\pmod4$,\, $(\frac{2}{p})=1$, $ \left( \frac{2}{p}\right)_{4}=-\left( \frac{p}{2}\right)_{4}=1 $ and $H$ is cyclic of order $2$ 	
		\item For $n=q_{1}q_{2}=59.83\equiv1\pmod4$,\, $(\frac{2}{q_1})=(\frac{2}{q_2})=-1$ and $H$ is cyclic of order $2$.
		For $n=q_{1}q_{2}=67.71\equiv1\pmod4$,\, $(\frac{2}{q_1})=-(\frac{2}{q_2})=-1$ and $H$ is cyclic of order $2$.
		\item For $n=pq=61.59\equiv3\pmod4$,\, $(\frac{2}{p})=(\frac{2}{q})=-1$ and $H$ is cyclic of order $2$.
		For $n=pq=61.47\equiv3\pmod4$,\, $(\frac{2}{p})=-(\frac{2}{q})=-1$ and $H$ is cyclic of order $2$.		
	\end{enumerate}
\end{exams}
\begin{them}\label{14}
	Let $\mathbb{K}=\mathbb{Q}(\sqrt{n\epsilon_{0}\sqrt{2}})$ be a real cyclic quartic number field,   $n$ an odd square-free positive  integer  and  $\epsilon_{0}$ the fundamental unit of  $k=\mathbb{Q}(\sqrt{2})$.
	The rank $r_2(H)$ equals $2$  if,  and only if $n$  takes one of the following forms.
	\begin{enumerate}[\rm1.]
		\item $n= p$ with $(\frac{2}{p})=1$ and $ \left( \frac{2}{p}\right)_{4} =\left( \frac{p}{2}\right)_{4} $
		\item $n= p_1p_2$ and:
		\begin{enumerate}[\rm a.]
			\item either $(\frac{2}{p_1})=(\frac{2}{p_2})=-1$
			\item or  $(\frac{2}{p_i})=- (\frac{2}{p_j})=-1 $ and $ \left( \frac{2}{p_j}\right)_{4} \neq\left( \frac{p_j}{2}\right)_{4}$ for $i\neq j$ in $\{1, 2\}$.
		\end{enumerate}
		\item $n= q_1q_2$ with $(\frac{2}{q_1})=(\frac{2}{q_2})=1$.
		\item $n= q_1q_2q_3$ and at most one of the  symbols $(\frac{2}{q_1})$,  $(\frac{2}{q_2})$,  $(\frac{2}{q_3})$ equals $1$.
		\item $n=p_1p_2q$ and   $(\frac{2}{p_1})=(\frac{2}{p_2})=-1$.
		\item $n=pq$ and $(\frac{2}{p})=1$.
		\item $n=pq_1q_2$ with $(\frac{2}{p})=-1$ and $(\frac{2}{q_1})=-1$ or  $(\frac{2}{q_2})=-1$
	\end{enumerate}
\end{them}
\begin{exams}~\
	\begin{enumerate}[\rm1.]
		\item For $n= p=881\equiv1\pmod4$,\,  $ \left( \frac{2}{p}\right)_{4} =\left( \frac{p}{2}\right)_{4}=1 $ and $H$ is  of type $ (4,4) $.	
		\item For $n=p_1p_2=877.997\equiv1\pmod4$, \, $(\frac{2}{p_1})=(\frac{2}{p_2})=-1$ and $H$ is  of type $ (2,2) $.	
		For $n=p_1p_2=941.977\equiv1\pmod4$, we have $(\frac{2}{p_1})=-(\frac{2}{p_2})=-1$ and $ \left( \frac{2}{p_2}\right)_{4} =-\left( \frac{p_2}{2}\right)_{4}=-1 $ , $H$ is of type $(2,2) $.
		\item For $n=q_1q_2=47.79\equiv1\pmod4$, we have $(\frac{2}{q_1})=(\frac{2}{q_2})=1$ and $H$ is  of type $ (4,4) $.
		\item For $n=q_1q_2q_3=67.83.43\equiv3\pmod4$, we have $(\frac{2}{q_1})=(\frac{2}{q_2})=(\frac{2}{q_3})=-1$ and $H$ is of type $ (2,2) $. For $n=q_1q_2q_3=67.83.47\equiv3\pmod4$, we have $(\frac{2}{q_1})=(\frac{2}{q_2})=-(\frac{2}{q_3})=-1$ and $H$ is of type $ (2,2) $.
		\item For $n=p_1p_2q=53.61.83\equiv3\pmod4$, we have $(\frac{2}{p_1})=(\frac{2}{p_2})=(\frac{2}{q_3})=-1$ and $H$ is of type $ (2,2) $.
		For $n=p_1p_2q_3=53.61.71\equiv3\pmod4$, we have $(\frac{2}{p_1})=(\frac{2}{p_2})=-(\frac{2}{q_3})=-1$ and $H$ is of type $ (2,2) $.
		\item For $n=pq=73.83\equiv3\pmod4$, \, $(\frac{2}{p})=-(\frac{2}{q})=1$ and $H$ is  of type $ (2,4) $.
		For $n=pq=73.79\equiv3\pmod4$, \, $(\frac{2}{p})=(\frac{2}{q})=1$ and $H$ is  of type $ (2,2) $.
		\item For $n=pq_1q_2=61.43.67\equiv1\pmod4$, we have $(\frac{2}{p})=(\frac{2}{q_1})=(\frac{2}{q_2})=-1$ and $H$ is of type $ (2,2) $.
		For $n=pq_1q_2=61.59.71\equiv1\pmod4$, we have $(\frac{2}{p})=(\frac{2}{q_1})=-(\frac{2}{q_2})=-1$ and $H$ is bicyclic and of type $ (2,2) $.
	\end{enumerate}
\end{exams}
\begin{them}\label{15}
	Let $\mathbb{K}=\mathbb{Q}(\sqrt{n\epsilon_{0}\sqrt{2}})$ be a real cyclic quartic number field,  $n$ an odd square-free positive  integer  and  $\epsilon_{0}$ the fundamental unit of  $k=\mathbb{Q}(\sqrt{2})$. 	The rank $r_2(H)$ equals $3$  if,  and only if $n$  takes one of the following forms.
	\begin{enumerate}[\rm1.]
		\item $n= p_1p_2$ and
        \begin{enumerate}[\rm a.]
        	\item either $(\frac{2}{p_1})=(\frac{2}{p_2})=1$ and $ \left( \frac{2}{p_{i}}\right)_{4} \neq \left( \frac{p_{i}}{2}\right)_{4} $ for at least one	$i\in\{1, 2\}$   .
        	\item or $(\frac{2}{p_i})=-(\frac{2}{p_j})=-1$ and $ \left( \frac{2}{p_{j}}\right)_{4} =\left( \frac{p_{j}}{2}\right)_{4} $ for $i\neq j\in\{1, 2\}$.
        \end{enumerate}
		\item $n= p_1p_2p_3$ and
		\begin{enumerate}[\rm a.]
			\item either $(\frac{2}{p_i})=-1$ for all $i\in\{1,  2,  3\}$.
			\item  or $(\frac{2}{p_i})=(\frac{2}{p_j})=-(\frac{2}{p_k})=-1$ and $ \left( \frac{2}{p_{k}}\right)_{4}\neq \left( \frac{p_{k}}{2}\right)_{4} $ for  $i$, $j$ and $k$ different two by two in $\{1, 2, 3\}$.
		\end{enumerate}
		\item $n= q_1q_2q_3$ and only one of the symbols $(\frac{2}{q_i})$,  $i\in\{1,  2,  3\}$,   equals $-1$.
		\item $n= q_1q_2q_3q_4$ and at most one of the symbols $(\frac{2}{q_i})$,  $i\in\{1,  2,  3, 4\}$,  is $1$.
		\item $n=p_1p_2p_3q$ and   $(\frac{2}{p_i})=-1$ for all $i\in\{1,  2,  3\}$.
		\item $n= pq_1q_2q_3$ with $(\frac{2}{p})=-1$ and  at most one of the symbols  $(\frac{2}{q_i})$, $i=1, 2, 3$, equals $1$.
		\item $n= p_1p_2q_1q_2$ with $(\frac{2}{p_i})=-1$ for all $i\in\{1,  2\}$ and  $((\frac{2}{q_1})$ or $(\frac{2}{q_2})=-1)$.
		\item $n= pq_1q_2$ and:
		\begin{enumerate}[\rm a.]
			\item either   $(\frac{2}{p })=1$ and $(\frac{2}{q_1})=-1$ or  $(\frac{2}{q_2})=-1 $  .
			\item or  $(\frac{2}{p })=-1$ and $(\frac{2}{q_i})=1$ for all $i\in\{1, 2\}$.
		\end{enumerate}
	  \item $n=p_1p_2q$ with $(\frac{2}{p_1})\neq(\frac{2}{p_2})$.
	\end{enumerate}
\end{them}
\begin{exams}~\
	\begin{enumerate}[\rm1.]
		\item For $n=p_1p_2=769.977\equiv1\pmod4$, we have $(\frac{2}{p_1})=(\frac{2}{p_2})=1$, $ \left( \frac{2}{p_{1}}\right)_{4} =- \left( \frac{p_{1}}{2}\right)_{4}=-1 $  and $ \left( \frac{2}{p_{2}}\right)_{4} =- \left( \frac{p_{2}}{2}\right)_{4}=-1 $ $H$ is of type $(2,2,4)$.
		For $n=p_1p_2=797.953\equiv1\pmod4$, we have $(\frac{2}{p_1})=-(\frac{2}{p_2})=-1$  and $ \left( \frac{2}{p_{2}}\right)_{4} =\left( \frac{p_{2}}{2}\right)_{4}=-1 $, $H$ is of type $(2,4,4)$.
		\item	For $n=p_1p_2p_3=37.53.61\equiv1\pmod4$, we have $(\frac{2}{p_i})=-1$ for all $i\in\{1,  2,  3\}$ and $H$ is of type $(2,2,2)$.
		For $n=p_1p_2p_3=53.61.89\equiv1\pmod4$, we have $(\frac{2}{p_1})=(\frac{2}{p_2})=-(\frac{2}{p_3})=-1$ and $ \left( \frac{2}{p_{3}}\right)_{4}=- \left( \frac{p_{3}}{2}\right)_{4}=-1 $, $H$ is of type $(2,2,2)$.
		\item For $n=q_1q_2q_3=71.79.67\equiv3\pmod4$, we have $(\frac{2}{q_1})=(\frac{2}{q_2})=-(\frac{2}{q_3})=1$  and $H$ is of type $(2,2,4)$.
		\item For $n=q_1q_2q_3q_4=59.67.83.43\equiv1\pmod4$, we have $(\frac{2}{q_1})=(\frac{2}{q_2})=(\frac{2}{q_3})=(\frac{2}{q_4})=-1$  and $H$ is of type $(2,2,2)$.
		For $n=q_1q_2q_3q_4=59.67.83.79\equiv1\pmod4$, we have $(\frac{2}{q_1})=(\frac{2}{q_2})=(\frac{2}{q_3})=-(\frac{2}{q_4})=-1$  and $H$ is of type $(2,2,2)$.
		\item For $n=p_1p_2p_3q=37.53.61.67\equiv3\pmod4$, we have $(\frac{2}{p_1})=(\frac{2}{p_2})=(\frac{2}{p_3})=(\frac{2}{q})=-1$  and $H$ is of type $(2,2,2)$.
		For $n=p_1p_2p_3q=37.53.61.71\equiv3\pmod4$, we have $(\frac{2}{p_1})=(\frac{2}{p_2})=(\frac{2}{p_3})=-(\frac{2}{q})=-1$  and $H$ is of type $(2,2,2)$.
		\item For $n=pq_1q_2q_3=61.67.83.59\equiv3\pmod4$, we have $(\frac{2}{p})=(\frac{2}{q_1})=(\frac{2}{q_2})=(\frac{2}{q_3})=-1$  and $H$ is of type $(2,2,2)$.
		For $n=pq_1q_2q_3=61.67.83.71\equiv3\pmod4$, we have $(\frac{2}{p})=(\frac{2}{q_1})=(\frac{2}{q_2})=-(\frac{2}{q_3})=-1$  and $H$ is of type $(2,2,2)$.
		\item For $n=p_1p_2q_1q_2=53.61.83.67\equiv1\pmod4$, we have $(\frac{2}{p_1})=(\frac{2}{p_2})=(\frac{2}{q_1})=(\frac{2}{q_2})=-1$  and $H$ is of type $(2,2,2)$.
		For $n=p_1p_2q_1q_2=53.61.83.79\equiv1\pmod4$, we have $(\frac{2}{p_1})=(\frac{2}{p_2})=(\frac{2}{q_1})=-(\frac{2}{q_2})=-1$  and $H$ is of type $(2,2,2)$.
		\item For $n=pq_1q_2=97.79.83\equiv1\pmod4$, we have $(\frac{2}{p})=(\frac{2}{q_1})=-(\frac{2}{q_2})=1$  and $H$ is of type $(2,2,2)$.
		For $n=pq_1q_2=61.47.71\equiv1\pmod4$, we have $(\frac{2}{p})=-(\frac{2}{q_1})=-(\frac{2}{q_2})=-1$  and $H$ is of type $(2,2,4)$.	
		\item For $n=p_1p_2q=61.73.83\equiv3\pmod4$, we have $(\frac{2}{p_1})=-(\frac{2}{p_2})=(\frac{2}{q})=-1$  and $H$ is of type $(2,2,4)$.	
	\end{enumerate}
\end{exams}

\end{document}